\documentclass[11pt,a4paper]{amsart}

\usepackage{amsmath,amssymb}
\usepackage{amsthm}
\usepackage{amssymb} % 数学記号（mathbbなど）を使うため
\usepackage{mathrsfs} % 筆記体（mathscr）を使うため
\usepackage{bm}
\usepackage{enumitem}
\usepackage[margin=25truemm]{geometry}
\usepackage{color}

\theoremstyle{plain} % イタリック体（定理、補題など）
\newtheorem{theorem}{Theorem}[section]
\newtheorem{proposition}[theorem]{Proposition}
\newtheorem{lemma}[theorem]{Lemma}
\newtheorem{corollary}[theorem]{Corollary}

\theoremstyle{definition} % 通常の書体（定義、例など）

\newtheorem{condition}[theorem]{Condition}

\theoremstyle{remark} % 備考など
\newtheorem{remark}[theorem]{Remark}

\numberwithin{equation}{section}

\allowdisplaybreaks

\newcommand{\Prob}{\mathbb{P}}
\newcommand{\E}{\mathbb{E}}

\newcommand{\bz}{\bm{0}}

\newcommand{\N}{\mathbf{N}}
\newcommand{\Z}{\mathbf{Z}}
\newcommand{\Nnn}{\mathbf{Z}_{\geq 0}}
\newcommand{\Np}{\mathbf{Z}_{\geq 1}}
\newcommand{\R}{\mathbf{R}}
\newcommand{\Rd}{\mathbf{R}^d}
\newcommand{\BN}{\mathbb{N}}
\newcommand{\CB}{\mathcal{B}}

\newcommand{\CI}{\mathcal{I}}
\newcommand{\CJ}{\mathcal{J}}
\newcommand{\CK}{\mathcal{K}}

\newcommand{\CN}{\mathcal{N}}
\newcommand{\CP}{\mathcal{P}}
\newcommand{\CQ}{\mathcal{Q}}
\newcommand{\CS}{\mathcal{S}}
\newcommand{\CU}{\mathcal{U}}
\newcommand{\CX}{\mathcal{X}}
\newcommand{\CY}{\mathcal{Y}}

\title[Maximum degree for random geometric graphs]{Maximum degree for random geometric graphs in the connectivity regime and beyond}
\author{Junpei Otsuka}
\address{Graduate School of Science and Technology, University of Tsukuba, Ibaraki, 305-8571, Japan.}
\email{jotsuka@math.tsukuba.ac.jp}

\begin{document}

\begin{abstract}
A random geometric graph is a random graph model whose vertex set is a binomial point process, where two vertices are connected by an edge if each lies within the ball of radius $r$ centered at the other.
In this paper, we investigate the asymptotic fluctuations of the threshold radius required for the maximum degree of the graph to reach a target value.
Assuming this target value grows with the number of vertices $n$ at an order of at least $\log n$ but strictly less than $(\log n)^d$, we show that these fluctuations follow a Gumbel distribution.
We also investigate the spatial configurations of vertices with extremely large degrees, and show that the point process marked by their degrees converges in distribution to a Poisson point process. 
\end{abstract}

% タイトルや著者名の出力
\maketitle

%\tableofcontents

\section{Introduction}

\subsection{Preliminaries}
Let $(X_i)_{i \in \N}$ be independent and identically distributed (i.i.d.) $\Rd$-valued random variables with a bounded density $f$, and let $f_{\max}$ denote its finite essential supremum.
For each $n \in \N$, we define the binomial point process $\CX_n := \sum_{i=1}^n \delta_{X_i}$, which is simple.
For $r > 0$, the random geometric graph $G(\CX_n; r)$ takes $\CX_n$ as its vertex set, where an edge connects any two distinct vertices if each is contained in the ball of radius $r$ centered at the other with respect to an arbitrarily fixed norm on $\Rd$.
For a comprehensive treatment of the model, see the monograph by Penrose~\cite{MR1986198}.
Much attention has been paid to the asymptotic behavior of $G(\CX_n; r_n)$ for a radii $(r_n)_{n \in \N}$ that is typically decreses to $0$.
We also consider random geometric graphs constructed on Poisson point processes.
Letting $N_{\lambda}$ be a Poisson random variable with mean $\lambda$ independent of $(X_i)_{i \in \N}$, we define $\CP_{\lambda} := \sum_{i=1}^{N_{\lambda}} \delta_{X_i}$, which forms a Poisson point process with intensity $\lambda f$.
Following the same construction as before, we define the graph $G(\CP_{\lambda}; r)$.

By a slight abuse of notation, we write $x \in \CX$ instead of $x \in \mathrm{supp}(\CX)$ for a point measure $\CX$ (see Section~\ref{Section.02}).
The maximum degree $\Delta_n$ of the graph $G(\CX_n; r_n)$ is defined as
\begin{equation*}
    \Delta_n := \max \{\deg_n(X) : X \in \CX_n\},
\end{equation*}
where $\deg_n(X)$ denotes the degree of $X$ in $G(\CX_n; r_n)$.
The almost sure limits of $\Delta_n$ have been extensively studied.
If $nr_n^d$ is bounded above by $n^{-\varepsilon}$ for some $\varepsilon > 0$, then $\Delta_n$ remains bounded, implying that no non-trivial almost sure limit for $\Delta_n$ exists.
In all other regimes satisfying $r_n \to 0$, non-trivial almost sure limits (or limits in probability) of $\Delta_n$ have been established.
However, when $nr_n^d$ is of order at least $\log n$, these results additionally require the bounded density $f$ to be compactly supported.
For more detail, see Penrose~\cite[Chapter~6]{MR1986198}.
Regarding the fluctuations of $\Delta_n$, if $nr_n^d$ is bounded above by $n^{-\varepsilon}$ for some $\varepsilon > 0$, then $\Delta_n$ concentrates on at most two bounded consecutive integers~\cite[Theorem~6.3]{MR1986198}.
M\"{u}ller~\cite{MR2501248} showed that for $nr_n^d = n^{o(1)}$ and $nr_n^d = o(\log n)$, a similar concentration holds on diverging consecutive integers.
Within his framework, he extended this concentration for $nr_n^d = o(\log n)$ to other key graph-theoretic quantities, including the clique number, chromatic number, and degeneracy.
The paper~\cite{MR2501248} also contains the fluctuation result for the regime where $nr_n^d$ is of order at least $\log n$, see Remark~\ref{Remark.01}.

For a sequence of positive integers $(k_n)_{n \in \N}$, the threshold radius $S_{k_n}(\CX_n)$ for the maximum degree to reach $k_n$ is defined as
\begin{equation*}
    S_{k_n}(\CX_n) := \inf \{r > 0 : (\text{the maximum degree of $G(\CX_n; r)$}) \geq k_n\}.
\end{equation*}
This is also known as the smallest $k_n$-nearest neighbor link~\cite{MR1986198}.
When $k_n$ is of order at least $\log n$, it has been shown that $S_{k_n}(\CX_n)$ admits a non-trivial almost sure limit, provided that $f$ is bounded and compactly supported~\cite[Chapter~6]{MR1986198}.
Furthermore, the fluctuation results in \cite{otsuka2026limit} show that $S_{k_n}(\CX_n)$ asymptotically follows a Weibull distribution for fixed $k_n$, and a Gumbel distribution when $k_n \to \infty$ with $k_n = o(\log n)$. 
This implies that a non-trivial limit in probability exists only in the latter case.
Note that the fluctuation result in the latter case were established not for an arbitrary bounded $f$, but specifically for a density $f$ satisfying Condition~\ref{Condition.01} of the present paper.
Under the same assumption, and based on M\"{u}ller's two-point concentration result for $\Delta_n$, the previous work~\cite{otsuka2026limit} investigated the point configurations realizing the maximum degree by analyzing those corresponding to the two consecutive integer degrees.
This revealed that the configurations exhibit compound Poisson behavior in the regime where $\Delta_n$ is bounded, and Poisson behavior where $\Delta_n$ diverges.

Regarding the threshold radius $M_{k_n}(\CX_n)$ for the \emph{minimum} degree, which is defined in a similar manner to $S_{k_n}(\CX_n)$, its fluctuations have been studied for fixed $k_n$.
The results for the uniform density $f$ on $[0,1]^d$ or the standard normal density $f$ can be found in \cite[Chapter~8]{MR1986198}.
Recently, Penrose and Yang~\cite{MR4993794} have addressed the case where the density $f$ is continuous and bounded away from $0$ on its support, assuming the support is connected and has a smooth ($d \geq 2$) or polygonal ($d = 2$) boundary.
The fluctuation of $M_{k_n}(\CX_n)$ is asymptotically Gumbel distributed, except under a special balance between $d$ and $k_n \equiv k$, in which case the convergence breaks down as multiple boundary effects of the support of $f$ compete with each other.

In the present paper, we show that $S_{k_n}(\CX_n)$ asymptotically follows a Gumbel distribution when $k_n$ is of order at least $\log n$ but strictly less than $(\log n)^d$.
Additionally, we examine the spatial configurations of vertices with exceptionally high degrees, verifying that they exhibit Poisson behavior.
The results in \cite{MR2501248} show that the fluctuations of $\Delta_n$ are no longer bounded in this regime, unlike the setting in \cite{otsuka2026limit}.
Therefore, it is natural to investigate the convergence of a point process marked by vertex degrees, which requires more involved computations than those in \cite{otsuka2026limit}.
As corollaries of these results, we can recover the fluctuation results for $\Delta_n$ in \cite[Section~5]{MR2501248}.
In fact, we obtain a slight refinement of the result for the regime where $nr_n^d = \Theta(\log n)$, see Remark~\ref{Remark.01}.

\subsection{Notation}\label{Section.02}
Let $\BN(\Rd)$ denote the space of all non-negative integer-valued locally finite measures on $\Rd$.
This space is equipped with the vague topology induced by the following mappings, along with the $\sigma$-algebra generated by this topology:
\begin{equation*}
    \mu \mapsto \int_{\Rd} f \,d\mu, \quad (f \in C_K^+(\Rd)).
\end{equation*}
Here, $C_K^+(\Rd)$ denotes the space of all non-negative, compactly supported continuous functions on $\Rd$.
Although the elements of $\BN(\Rd)$ are technically counting measures, we frequently encounter situations where we want to consider the set inclusion of the point configurations naturally represented by them.
For notational convenience, we define the inclusion $\CX \subset \CY$ between $\CX, \CY \in \BN(\Rd)$ by $\mathrm{supp}(\CX) \subset \mathrm{supp}(\CY)$.
This is a natural definition, since we deal exclusively with simple counting measures in the present work.
For the same reason, we simply write $X \in \CX$ instead of $X \in \mathrm{supp}(\CX)$.

Let $W_0$ denote the principal branch of the Lambert $W$ function, that is, the inverse function of $te^t$ for $t \in [-1,\infty)$.
For $\alpha \in (0,\infty)$, let $Z(\alpha)$ denote a Gumbel random variable scaled by $\alpha$ with $P[Z(\alpha) \leq x] = \exp(-\exp(-\alpha^{-1}x))$.
In particular, $Z(1)$ is a standard Gumbel variable, and $Z(\alpha)$ has the same distribution as $\alpha Z(1)$.
The letters $c, c_0, c_1, c_2$ represent non-random positive constants independent of $n$, and their values may change from line to line.
Throughout this paper, $\theta$ denotes the volume of a $d$-dimensional unit ball.
We omit explicitly stating its dependence on $d$, since it does not affect the calculations.

\section{Main Results}
We determine the limiting distribution of the threshold radius associated with the degree parameters $(k_n)_{n \in \N}$ corresponding to the connectivity and super-connectivity regimes.
We also characterize the limiting spatial point process of the point configuration in these regimes, taking into account the fluctuations of degrees as information marked on each point.

\subsection{Connectivity regime}
In this section, we discuss results concerning the so-called connectivity regime, which corresponds to the condition $nr_n^d = \Theta(\log n)$.

\begin{theorem}\label{Theorem.01}
    Let $d \geq 2$.
    Assume that a sequence of integers $(k_n)_{n \in \N}$ satisfying $k_n/\log n \to b$ for some $b \in (0,\infty)$.
    Suppose that $f$ is the uniform density function on a $d$-dimensional unit cube.
    Define the positive constants
    \begin{equation*}
        a_n := -W_0 \bigg(-\frac{1}{en^{1/k_n}}\bigg), \quad b_n := \frac{1}{en^{1/k_n}} W_0' \bigg(-\frac{1}{en^{1/k_n}}\bigg)
    \end{equation*}
    and let their respective limits as $n \to \infty$ be denoted by $a_{\infty} := -W_0(-e^{-1-1/b})$ and $b_{\infty} := e^{-1-1/b} W_0'(-e^{-1-1/b})$.
    Then, it holds that
    \begin{equation*}
        -n\theta S_{k_n}(\CX_n)^d + a_n k_n + \frac{1}{2} b_n \log k_n + b_{\infty} \log ((1-a_{\infty}) \sqrt{2\pi}) \xrightarrow{d} Z(b_{\infty})
    \end{equation*}
    as $n \to \infty$.
\end{theorem}

We can extend the above result to densities having a strict maximum at the origin and sufficient smoothness nearby.
More precisely, we impose the following condition:
\begin{condition}\label{Condition.01}
    The density $f$ is the uniform density on $[-1/2,1/2]^d$ or satisfies following two conditions:
    \begin{itemize}
        \item For some $\rho_0 > 0$ and $m \in \Np$, $f$ is of class $C^{2m}$ on $B(\bz ; \rho_0)$.
        Further, the origin $\bz$ is the unique global maximum point of $f$ with $f(\bz) = f_{\max}$.
        In addition, $\min \{\ell > 0 : \partial^{\ell} f(\bz) \neq 0\} = 2m$ and the $2m$-th derivative $\partial^{2m} f(\bz)$ is a negative definite $2m$-linear form on $\Rd$.
        \item There exist $\rho_1 \in (0,\rho_0)$ and $\varepsilon_0 > 0$ such that $f(x) \leq (1-\varepsilon_0) f_{\max}$ for all $x \in \Rd \setminus B(\bz; \rho_1)$.
    \end{itemize}
\end{condition}

For simplicity, we write $\partial^{2m} f(\bz)[z^{\otimes 2m}]$ instead of $\partial^{2m} f(\bz)[z, \cdots, z]$ in what follows.
\begin{theorem}\label{Theorem.02}
    Let $d \geq 2$.
    Let $(k_n)_{n \in \N}$ be a sequence of integers satisfying $k_n/\log n \to b$ for some $b \in (0,\infty)$.
    Suppose that $f$ is a non-uniform density function satisfying Condition~\ref{Condition.01}.
    Then, it holds that
    \begin{equation*}
        -n\theta f_{\max} S_{k_n}(\CX_n)^d + a_n k_n + \bigg(\frac{1+(d/m)}{2}\bigg) b_n \log k_n + b_{\infty} \log \frac{(1-a_{\infty})^{1+d/(2m)} \sqrt{2\pi}}{\gamma f_{\max}^{1+d/(2m)}} \xrightarrow{d} Z(b_{\infty})
    \end{equation*}
    as $n \to \infty$, where we define the constant $\gamma := \int_{\Rd} \exp(\partial^{2m}f(\bz)[z^{\otimes 2m}]/((2m)!)) \,dz$.
\end{theorem}

For a chosen sequence of radii $(r_n)_{n \in \N}$, we write $\Phi_n \in \BN(\Rd \times \Z)$ for the degree-marked point configuration of vertices in $G(\CX_n; r_n)$, spatially scaled according to the chosen density function $f$, with the degree marks centered at $k_n$.
The precise definition will be given within the statement of the following theorem.
Let $W_{j,n}'$ denote the number of vertices with degree $j$ in $G(\CP_n; r_n)$.

\begin{theorem}\label{Theorem.03}
    Let $d \geq 2$.
    Let $(k_n)_{n \in \N}$ be a sequence of integers satisfies $k_n/\log n \to b$ for some $b \in (0,\infty)$.
    Suppose that $f$ is a density function satisfying Condition~\ref{Condition.01}.
    Let $(r_n)_{n \in \N}$ be a sequence of radii satisfies $\E[\sum_{j=k_n}^\infty W_{j,n}'] \to \beta$ for some $\beta \in (0,\infty)$.
    Then, in the uniform case, $\Phi_n := \sum_{X \in \CX_n} \delta_{(X,~ \deg_n(X) - k_n)}$ converges in distribution, with respect to the vague topology, to a Poisson point process $\CQ$ on $\Rd \times \Z$ with intensity measure $\Lambda$ uniquely determined by
    \begin{equation*}
        \Lambda(A \times B) := \beta \mathrm{Leb}(A \cap [-1/2,1/2]^d) \cdot G_b(B) \quad (A \in \CB(\Rd), B \subset \Z),
    \end{equation*}
    where $G_b$ denotes the measure on $\Z$ determined by its values on singletons:
    \begin{equation*}
        G_b(\{m\}) := (-W_0(-e^{-1-1/b}))^m (1+W_0(-e^{-1-1/b}))
    \end{equation*}
    for each $m \in \Z$.
    In the non-uniform case, $\Phi_n := \sum_{X \in \CX_n} \delta_{(k_n^{1/(2m)} X,~ \deg_n(X) - k_n)}$ converges in distribution, with respect to the vague topology, to a Poisson point process $\CQ$ on $\Rd \times \Z$ with intensity measure $\Lambda$ uniquely determined by
    \begin{equation*}
        \Lambda(A \times B) := \beta \gamma^{-1} \int_A \exp\bigg(\frac{\partial^{2m}f(\bz)[z^{\otimes 2m}]}{(2m)!}\bigg) \,dz \cdot G_b(B) \quad (A \in \CB(\Rd), B \subset \Z).
    \end{equation*}
\end{theorem}

\subsection{Super-connectivity regime}
In this section, we focus on the super-connectivity regime, which is characterized by the condition $nr_n^d/\log n \to \infty$.

\begin{theorem}\label{Theorem.04}
    Let $d \geq 2$.
    Let $(k_n)_{n \in \N}$ be a sequence of integers satisfies $\log n \ll k_n \ll (\log n)^d$.
    Suppose that $f$ is a density function satisfying Condition~\ref{Condition.01}.
    Then, there exist two sequences of real numbers $(a_n)_{n \in \N}$ and $(b_n)_{n \in \N}$ such that
    \begin{equation*}
        \frac{-n\theta f_{\max} S_{k_n}(\CX_n)^d + a_n}{b_n} \xrightarrow{d} Z(1)
    \end{equation*}
    holds as $n \to \infty$.
    Furthermore, we can choose $b_n := \sqrt{k_n/(2\log n)}$ and the sequence $(a_n)_{n \in \N}$ exhibits the following asymptotic behavior up to the third order as $n \to \infty$:
    \begin{equation*}
        a_n = k_n - \sqrt{2k_n \log n} + \bigg(\frac{2}{3}\log n + \sqrt{\frac{k_n}{8\log n}} \bigg\{\bigg(1+\frac{d}{2m}\bigg)\log \log n + \frac{d}{4m} \log k_n\bigg\}\bigg) (1+o(1)).
    \end{equation*}
    Here, in the uniform case, the degree of flatness near the maximum point is treated as $m = \infty$; thus, by convention, we set $m^{-1} = 0$.
\end{theorem}

\begin{remark}
    As a simple example in the uniform case of Theorem~2.5, let us choose $k_n = \lfloor (\log n)^3 \rfloor$ for $d \geq 4$.
    Then, we can explicitly set
    \begin{equation*}
        a_n = (\log n)^3 - \sqrt{2} (\log n)^2 + \frac{\log \log n}{2\sqrt{2}} \log n + \bigg(\frac{\log (2\sqrt{\pi})}{\sqrt{2}} + \frac{2}{3}\bigg) \log n,
    \end{equation*}
    which follows from the Taylor expansion (\ref{equation.07}).
    If we choose $k_n = \lfloor (\log n)^2 \rfloor$ for $d \geq 3$, the determination of $a_n$ becomes more complicated, requiring the Taylor expansion up to the third order.
\end{remark}

For a pair of real sequences $(a_n)_{n \in \N}$ and $(b_n)_{n \in \N}$, we write $\Phi_n \in \BN(\Rd \times \R)$ for the degree-marked point configuration of vertices in $G(\CX_n; r_n)$, spatially scaled according to the chosen density function $f$, with the degree marks centered at $a_n$ and scaled by $b_n$.
The precise definition will be given within the statement of the following theorem:

\begin{theorem}\label{Theorem.06}
    Let $d \geq 2$.
    Let $(r_n)_{n \in \N}$ be a sequence of radii satisfies $\log n \ll n\theta r_n^d \ll (\log n)^d$.
    Suppose that $f$ is the uniform density function $\mathbf{1}_{[-1/2,1/2]^d}$.
    Then, there exist two sequences of real numbers $(a_n)_{n \in \N}$ and $(b_n)_{n \in \N}$ such that $\Phi_n := \sum_{X \in \CX_n} \delta_{(X,~ (\deg_n(X) - a_n)/b_n)}$ converges in distribution, with respect to the vague topology, to a Poisson point process $\CQ$ on $\Rd \times \R$ with intensity measure $\Lambda$ uniquely determined by
    \begin{equation*}
        \Lambda(A \times B) := \mathrm{Leb}(A \cap [-1/2,1/2]^d) \cdot E(B) \quad (A \in \CB(\Rd), B \in \CB(\R)).
    \end{equation*}
    Here, $E$ denotes the measure on $\R$ having the density function $e^{-x}$ with respect to the Lebesgue measure on the real line.
    Furthermore, for $\mu_n := n\theta r_n^d$, we can choose $b_n := \sqrt{\mu_n/(2\log n)}$, and the sequence $(a_n)_{n \in \N}$ exhibits the following asymptotic behavior as $n \to \infty$:
    \begin{equation*}
        a_n = \mu_n + \sqrt{2\mu_n \log n} + \bigg(\frac{1}{3}\log n - \sqrt{\frac{\mu_n}{8\log n}} \log \log n\bigg) (1+o(1)).
    \end{equation*}
\end{theorem}

\begin{theorem}\label{Theorem.07}
    Let $d \geq 2$.
    Let $(r_n)_{n \in \N}$ be a sequence of radii satisfies $\log n \ll n\theta r_n^d \ll (\log n)^d$.
    Suppose that $f$ is a non-uniform density function satisfying Condition~\ref{Condition.01}.
    Then, there exist $(a_n)_{n \in \N}$ and $(b_n)_{n \in \N}$ such that $\Phi_n := \sum_{X \in \CX_n} \delta_{(k_n^{1/(2m)} X,~ (\deg_n(X) - a_n)/b_n)}$ converges in distribution, with respect to the vague topology, to a Poisson point process $\CQ$ on $\Rd \times \R$ with intensity measure $\Lambda$ uniquely determined by
    \begin{equation*}
        \Lambda(A \times B) := \gamma^{-1} \int_A \exp\bigg(\frac{\partial^{2m}f(\bz)[z^{\otimes 2m}]}{(2m)!}\bigg) \,dz \cdot E(B) \quad (A \in \CB(\Rd), B \in \CB(\R)).
    \end{equation*}
    Furthermore, for $\mu_n := n\theta r_n^d$, we can choose $b_n := \sqrt{f_{\max} \mu_n/(2\log n)}$, and the sequence $(a_n)_{n \in \N}$ exhibits the following asymptotic behavior up to the third order as $n \to \infty$:
    \begin{equation*}
        \begin{aligned}
            a_n &= f_{\max} \mu_n + \sqrt{2f_{\max} \mu_n \log n}\\
            &\quad + \bigg(\frac{1}{3}\log n - \sqrt{\frac{f_{\max} \mu_n}{8\log n}} \bigg\{\bigg(1+\frac{d}{2m}\bigg)\log \log n + \frac{d}{4m} \log \mu_n\bigg\}\bigg) (1+o(1)).
        \end{aligned}
    \end{equation*}
\end{theorem}

\subsection{Corollaries of main theorems}
This section is devoted to stating the corollaries that follow from the main results.
The first two concern the limits in probability for both the threshold radius of the maximum degree and the maximum degree itself.

\begin{corollary}\label{Corollary.01}
    Let $d \geq 2$.
    Let $(k_n)_{n \in \N}$ be a sequence of integers satisfies $k_n/\log n \to b$ for some $b \in (0,\infty]$ and $k_n = o((\log n)^d)$.
    Suppose that $f$ is a density function satisfying Condition~\ref{Condition.01}.
    Then, as $n \to \infty$,
    \begin{equation*}
        \frac{n\theta S_{k_n}(\CX_n)^d}{k_n} \to \frac{-W_0(-e^{-1-1/b})}{f_{\max}}, \quad \text{in probability}.
    \end{equation*}
\end{corollary}

\begin{corollary}\label{Corollary.02}
    Let $d \geq 2$.
    Let $(r_n)_{n \in \N}$ be a sequence of radii satisfies $n\theta r_n^d/\log n \to b$ for some $b \in (0,\infty]$ and $n\theta r_n^d = o((\log n)^d)$.
    Suppose that $f$ is a density function satisfying Condition~\ref{Condition.01}.
    Then, as $n \to \infty$,
    \begin{equation*}
        \frac{\Delta_n}{n\theta r_n^d} \to f_{\max} H_+^{-1}\bigg(\frac{1}{f_{\max} b}\bigg), \quad \text{in probability},
    \end{equation*}
    where $H_+^{-1}$ denotes the inverse function of $H(t) := 1 - t + t \log t ~ (t \geq 1)$.
\end{corollary}

\begin{remark}
    Previously, the almost sure limits of $\Delta_n$ ($S_{k_n}(\CX_n)$) in the regime where $nr_n^d$ (resp., $k_n$) is of order at least $\log n$ were established only for compactly supported bounded densities.
    In contrast, restricted to the regime where $nr_n^d$ (resp., $k_n$) is of order at least $\log n$ but less than $(\log n)^d$, these corollaries establish the limits in probability for general bounded densities without requiring compact support, assuming only regularity and non-degeneracy near the maximum.
\end{remark}

Next, we state the results on the fluctuations of the maximum degree.

\begin{corollary}\label{Corollary.03}
    Let $d \geq 2$.
    Let $(k_n)_{n \in \N}$ be a sequence of integers satisfies $k_n/\log n \to b$ for some $b \in (0,\infty)$.
    Suppose that $f$ is a density function satisfying Condition~\ref{Condition.01}.
    Let $(r_n)_{n \in \N}$ be a sequence of radii satisfies $\E[\sum_{j=k_n}^\infty W_{j,n}'] \to \beta$ for some $\beta \in (0,\infty)$.
    Then, for a integer-valued random variable $Z_b$ satisfying $\Prob[Z_b \leq K] = \exp(-(-W_0(-e^{-1-1/b}))^{K+1}) ~ (K \in \Z)$, it holds that
    \begin{equation*}
        \Delta_n - k_n \xrightarrow{d} Z_b.
    \end{equation*}
\end{corollary}

\begin{corollary}\label{Corollary.04}
    Let $d \geq 2$.
    Let $(r_n)_{n \in \N}$ be a sequence of radii satisfies $\log n \ll n\theta r_n^d \ll (\log n)^d$.
    Suppose that $f$ is a density function satisfying Condition~\ref{Condition.01}.
    Then, there exist two sequences of real numbers $(a_n)_{n \in \N}$ and $(b_n)_{n \in \N}$ such that
    \begin{equation*}
        \frac{\Delta_n - a_n}{b_n} \xrightarrow{d} Z(1)
    \end{equation*}
    holds.
    Furthermore, for $\mu_n := n\theta r_n^d$, we can choose $b_n := \sqrt{f_{\max} \mu_n/(2\log n)}$, and the sequence $(a_n)_{n \in \N}$ exhibits the following asymptotic behavior up to the third order as $n \to \infty$:
    \begin{equation*}
        \begin{aligned}
            a_n &= f_{\max} \mu_n + \sqrt{2f_{\max} \mu_n \log n}\\
            &\quad + \bigg(\frac{1}{3}\log n - \sqrt{\frac{f_{\max} \mu_n}{8\log n}} \bigg\{\bigg(1+\frac{d}{2m}\bigg)\log \log n + \frac{d}{4m} \log \mu_n\bigg\}\bigg) (1+o(1)).
        \end{aligned}
    \end{equation*}
    Here, in the uniform case, the degree of flatness near the maximum point is treated as $m = \infty$; thus, by convention, we set $m^{-1} = 0$.
\end{corollary}

\begin{remark}\label{Remark.01}
    We now look at the fluctuation results on $\Delta_n$ obtained above in light of those in \cite{MR2501248}.
    The results in \cite[Section~5]{MR2501248} are the following: in the case where $f$ is the uniform density on a $d$-dimensional unit cube,
    \begin{enumerate}
        \item[(1)] whenever $\log n \ll nr_n^d \ll (\log n)^d$, the suitably translated and normalized $\Delta_n$ converges in distribution to a standard Gumbel variable,
        \item[(2)] whenever $nr_n^d = \Theta(\log n)$, the fluctuations of $\Delta_n$ are unbounded.
    \end{enumerate}
    Corollary~\ref{Corollary.04} above recovers (1) and extends it to general densities satisfying Condition~\ref{Condition.01}.
    Moreover, Corollary~\ref{Corollary.03} demonstrates that for $nr_n^d = \Theta(\log n)$, under a suitable choice of $r_n$, the unbounded fluctuation of $\Delta_n$ indeed converges weakly to a specific probability distribution, which is referred to as the discrete Gumbel distribution in \cite{MR4216888}.
\end{remark}

\section{Proofs for the connectivity regime}
\subsection{Threshold radius}\label{Section.01}
We begin with the uniform case.
Given a sequence of radii $(r_n)_{n \in \N}$ and an integer $j \geq 0$, let $W_{j,n}$ and $W_{j,n}'$ denote the number of vertices with degree $j$ in $G(\CX_n; r_n)$ and $G(\CP_n; r_n)$, respectively.
Then, note that the equivalence
\begin{equation}
    S_{k_n}(\CX_n) > r_n \iff \sum_{j=k_n}^{n-1} W_{j,n} = 0 \label{equation.06}
\end{equation}
holds.
Therefore, characterizing the fluctuation of $S_{k_n}(\CX_n)$ reduces to it that of $\sum_{j=k_n}^{n-1} W_{j,n}$.
To achieve this, we find an explicit form of $(r_n)_{n \in \N}$ that satisfies $\E[\sum_{j=k_n}^{n-1} W_{j,n}] \to \beta \in (0,\infty)$ and show that $\sum_{j=k_n}^{n-1} W_{j,n}$ is approximately Poisson.
For computational simplicity, we first analyze a Poissonized version $\sum_{j=k_n}^\infty W_{j,n}'$ rather than $\sum_{j=k_n}^{n-1} W_{j,n}$.
Based on the a.s. limit result in Penrose~\cite[Theorem~6.11]{MR1986198}, we search for $(r_n)_{n \in \N}$ that satisfies the following conditions:
\begin{equation}\label{equation.01}
    \frac{n\theta r_n^d}{k_n} \to C_0 \in (0,1) \quad \text{and} \quad E\bigg[\sum_{j=k_n}^\infty W_{j,n}'\bigg] \sim n\sum_{j \geq k_n} \frac{(n\theta r_n^d)^j}{j!} e^{-n\theta r_n^d} \to \beta.
\end{equation}
We show that there exists such a sequence $(r_n)_{n \in \N}$ and that $C_0$ can be explicitly determined.
For $m \geq 1$ and $\ell \geq 0$, define $(m+\ell)_{\ell} := \prod_{i=1}^{\ell} (m+i)$.
Then,
\begin{equation*}
    n\sum_{j \geq k_n} \frac{(n\theta r_n^d)^j}{j!} e^{-n\theta r_n^d} = n \frac{(n\theta r_n^d)^{k_n}}{k_n!} \sum_{j \geq 0} \frac{(n\theta r_n^d)^j}{(k_n+j)_j} e^{-n\theta r_n^d}
\end{equation*}
holds.
If $(r_n)_{n \in \N}$ satisfies $n\theta r_n^d/k_n \to C_0 \in (0,1)$, then for some $a \in (0,1)$ and for all sufficiently large $n \in \N$, we have $\sup_{j \geq 1} (n\theta r_n^d)^j/(k_n+j)_j \leq a$.
Thus, for any $\varepsilon > 0$, there exists $J \in \N$ such that $\sum_{j \geq J+1} (n\theta r_n^d)^j/(k_n+j)_j \leq \varepsilon$ for all sufficiently large $n \in \N$.
Moreover, we have
\begin{equation*}
    \sum_{j=0}^J \frac{(n\theta r_n^d)^j}{(k_n+j)_j} \xrightarrow{n \to \infty} \sum_{j=0}^J C_0^j
\end{equation*}
and thus by choosing $J \in \N$ such that $\sum_{j \geq J+1} C_0^j \leq \varepsilon$, we obtain
\begin{equation*}
    \limsup_{n \to \infty} \bigg|\sum_{j=0}^\infty \frac{(n\theta r_n^d)^j}{(k_n+j)_j} - \sum_{j=0}^\infty C_0^j\bigg| \leq 2\varepsilon.
\end{equation*}
Since $\varepsilon > 0$ is arbitrary, we conclude that $\sum_{j=0}^\infty (n\theta r_n^d)^j/(k_n+j)_j \to (1-C_0)^{-1}$ as $n \to \infty$.
Thus, it holds by Stirling's approximation that
\begin{equation*}
    \begin{aligned}
        n \frac{(n\theta r_n^d)^{k_n}}{k_n!} \sum_{j \geq 0} \frac{(n\theta r_n^d)^j}{(k_n+j)_j} e^{-n\theta r_n^d} &\sim n \frac{(n\theta r_n^d)^{k_n}}{k_n!} \sum_{j \geq 0} \frac{1}{1-C_0} e^{-n\theta r_n^d}\\
        &\sim \frac{n}{\sqrt{2\pi k_n}} \bigg(\frac{en\theta r_n^d}{k_n}\bigg)^{k_n} \frac{1}{1-C_0} e^{-n\theta r_n^d}.
    \end{aligned}
\end{equation*}
To make the expectation asymptotic to $\beta$, it suffices to choose $(r_n)_{n \in \N}$ such that
\begin{equation*}
    \frac{n}{\sqrt{2\pi k_n}} \bigg(\frac{en\theta r_n^d}{k_n}\bigg)^{k_n} \frac{1}{1-C_0} e^{-n\theta r_n^d} = \beta.
\end{equation*}
Since $W_0$ is the inveerse function of $te^t$ for $t \in [-1,\infty)$, we have
\begin{equation*}
    \frac{n\theta r_n^d}{k_n} = -W_0\bigg(-\frac{1}{e} \bigg(\frac{\beta (1-C_0)\sqrt{2\pi k_n}}{n}\bigg)^{1/k_n}\bigg).
\end{equation*}
Assuming that $k_n / \log n \to b$, it follows from the continuity of $W_0$ that
\begin{equation*}
    \frac{n\theta r_n^d}{k_n} \xrightarrow{n \to \infty} -W_0\bigg(-\frac{1}{e^{1+1/b}}\bigg) \in (0,1).
\end{equation*}
Hence, we can obtain such a sequence $(r_n)_{n \in \N}$ that meets \eqref{equation.01} by setting
\begin{equation}\label{equation.02}
    \frac{n\theta r_n^d}{k_n} = -W_0\bigg(-\frac{1}{e} \bigg(\frac{\beta (1+W_0(-e^{-1-1/b}))\sqrt{2\pi k_n}}{n}\bigg)^{1/k_n}\bigg)
\end{equation}
for each $n \in \N$.

In what follows, we show that $\sum_{j=k_n}^{n-1} W_{j,n}$ can be approximated by a Poisson random variable.
We use the following lemma:
\begin{lemma}[{\cite[Theorem~6.7]{MR1986198}}]\label{Lemma.01}
    Assume that $f$ is bounded.
    Let $A \subset \Nnn$, $r > 0$ and $\lambda > 0$.
    Define $\CP_{\lambda}^x := \CP_{\lambda} + \delta_x$ and 
    \begin{equation*}
        \begin{aligned}
            &I_1 := \lambda^2 \int_{\Rd} F(dx) \int_{B(x; 3r)} F(dy) \Prob[\CP_{\lambda}(B(x; r)) \in A] \Prob[\CP_{\lambda}(B(y; r)) \in A],\\
            &I_2 := \lambda^2 \int_{\Rd} F(dx) \int_{B(x; 3r)} F(dy) \Prob[\{\CP_{\lambda}^y(B(x; r)) \in A\} \cap \{\CP_{\lambda}^x(B(y; r)) \in A\}].
        \end{aligned}
    \end{equation*}
    Then, for $W_{A,\lambda}' := \sum_{j \in A} W_{j,\lambda}'$, the number of vertices with degree in $A$ in $G(\CP_{\lambda}; r)$, it holds that
    \begin{equation*}
        d_{\mathrm{TV}}(W_{A,\lambda}', \mathrm{Po}(\E[W_{A,\lambda}'])) \leq 3(I_1 + I_2),
    \end{equation*}
    where $d_{\mathrm{TV}}$ denotes the total variation distance.
\end{lemma}

\begin{remark}
    In \cite[Theorem~6.7]{MR1986198}, the proof is provided under the assumption that the density $f$ of $X_1$ is almost everywhere continuous.
    We can dispense with this assumption by applying the Lebesgue density theorem, and thus we state Lemma~\ref{Lemma.01} for a general bounded density $f$.
    The proof for this general case is essentially contained in the proof of \cite[Proposition~4.7]{otsuka2026limit}, to which the interested reader is referred.
\end{remark}

\begin{lemma}\label{Lemma.02}
    Suppose that $d \geq 2$ and $(\lambda(n))_{n \in \N}$ satisfies $n-n^{3/4} \leq \lambda(n) \leq n+n^{3/4}$ for each $n \in \N$.
    Let $f$ be the density function $\mathbf{1}_{[-1/2,1/2]^d}$.
    If a sequence $(r_n)_{n \in \N}$ satisfies \eqref{equation.01}, then for any two integers $0 \leq J_1 < J_2 \leq \infty$, it holds that $\sum_{j = k_n+J_1}^{k_n+J_2} W_{j,\lambda(n)}' \xrightarrow{d} \mathrm{Po}(\beta (1-C_0)\sum_{j=J_1}^{J_2} C_0^j)$.
\end{lemma}

\begin{proof}
    First, we assume that $J_2 < \infty$.
    Set $A_n := \{\ell \in \Nnn : k_n+J_1 \leq \ell \leq k_n+J_2\}$.
    Note that
    \begin{equation*}
        \E[W_{A_n,\lambda(n)}'] \xrightarrow{n \to \infty} \beta (1-C_0)\sum_{j=J_1}^{J_2} C_0^j
    \end{equation*}
    holds since $(\lambda(n)/n)^{k_n} \to 1$ as $n \to \infty$.
    By Lemma~\ref{Lemma.01}, we have
    \begin{equation*}
        \begin{aligned}
            &d_{\mathrm{TV}}(W_{A_n,\lambda(n)}', \mathrm{Po}(\E[W_{A_n,\lambda(n)}'])) \leq 3(I_1(n)+I_2(n)),\\
            &I_1(n) := \lambda(n)^2 \int_{\Rd} F(dx) \int_{B(x; 3r_n)} F(dy) \Prob[\CP_{\lambda(n)}(B_x^n) \in A_n] \Prob[\CP_{\lambda(n)}(B_y^n) \in A_n],\\
            &I_2(n) := \lambda(n)^2 \int_{\Rd} F(dx) \int_{B(x; 3r_n)} F(dy) \Prob[\{\CP_{\lambda(n)}^y(B_x^n) \in A_n\} \cap \{\CP_{\lambda(n)}^x(B_y^n) \in A_n\}],
        \end{aligned}
    \end{equation*}
    where $B_x^n := B(x; r_n)$.
    We show that $I_1(n), I_2(n) \to 0$.

    By the definition of $I_1(n)$ and the condition on $(\lambda(n))_{n \in \N}$, we obtain
    \begin{equation*}
        I_1(n) \leq \bigg(\lambda(n) \sum_{j \in A_n}\frac{(\lambda(n)\theta r_n^d)^j}{j!} e^{-\lambda(n)\theta r_n^d}\bigg)^2 3^d r_n^d \leq cr_n^d \to 0.
    \end{equation*}
    In what follows, we show that $I_2(n) \to 0$.
    In the definition of $I_2(n)$, we carry out the change of variables $x = r_nx'$, $y = r_ny'$.
    Then, for the function $h_n$ defined by
    \begin{equation*}
        h_n(x',y') := \Prob[\{\CP_{\lambda(n)r_n^d}^{y'}(B(x'; 1)) \in A_n\} \cap \{\CP_{\lambda(n)r_n^d}^{x'}(B(y'; 1)) \in A_n\}],
    \end{equation*}
    where $\CP_{\lambda(n)r_n^d}$ denotes a Poisson point process with intensity $\lambda(n)r_n^d \mathbf{1}_{[-1/(2r_n),1/(2r_n)]^d}$, it follows that
    \begin{equation*}
        I_2(n) \leq \lambda(n)^2 r_n^{2d} \int_{[-1/(2r_n),1/(2r_n)]^d} \int_{B(x'; 3)} h_n(x',y') \,dy'dx'.
    \end{equation*}
    For readability, we henceforth drop the primes from $x'$ and $y'$.
    With this convention, we write $B_x := B(x; 1)$ and $F_n(\cdot) := \mathrm{Leb}(\cdot \cap [-1/(2r_n),1/(2r_n)]^d)$.

    If $y \in B_x \setminus \{x\}$, then
    \begin{equation*}
        \begin{aligned}
            h_n(x,y) &= \Prob[\CP_{\lambda(n)r_n^d}(B_x)+1 \in A_n \mid \CP_{\lambda(n)r_n^d}(B_y)+1 \in A_n] \cdot \Prob[\CP_{\lambda(n)r_n^d}(B_y)+1 \in A_n]\\
            &\leq cn^{-1}\Prob[\CP_{\lambda(n)r_n^d}(B_x)+1 \in A_n \mid \CP_{\lambda(n)r_n^d}(B_y)+1 = k_n+J_2]
        \end{aligned}
    \end{equation*}
    holds.
    Regarding this probability, it can be shown that conditionally on $\CP_{\lambda(n)r_n^d}(B_y) = k_n+J_2-1$,
    \begin{itemize}
        \item the number of points in $B_x \cap B_y$ obeys the binomial distribution $\mathrm{Bi}(k_n+J_2-1, F_n(B_x \cap B_y)/F_n(B_y))$,
        \item the number of points in $B_x \setminus B_y$ obeys $\mathrm{Po}(\lambda(n)r_n^d F_n(B_x \setminus B_y))$,
    \end{itemize}
    and these are independent of each other.
    Therefore, we can rewrite the above as
    \begin{equation*}
        \begin{aligned}
            &\Prob[\mathrm{Bi}(k_n+J_2-1, F_n(B_x \cap B_y)/F_n(B_y)) + \mathrm{Po}(\lambda(n)r_n^d F_n(B_x \setminus B_y)) \in A_n]\\
            &= \Prob[k_n+J_2-1 - \mathrm{Bi}(k_n+J_2-1, F_n(B_y \setminus B_x)/F_n(B_y)) + \mathrm{Po}(\lambda(n)r_n^d F_n(B_x \setminus B_y)) \in A_n].
        \end{aligned}
    \end{equation*}
    Set $F_{y \setminus x}^n := F_n(B_y \setminus B_x)/F_n(B_y)$.
    If both of following two events
    \begin{equation*}
        \begin{aligned}
            &\{\mathrm{Bi}(k_n+J_2-1, F_{y \setminus x}^n) > k_n F_{y \setminus x}^n (1+C_0)/2\},\\
            &\{\mathrm{Po}(\lambda(n)r_n^d F_n(B_x \setminus B_y)) < k_n F_{y \setminus x}^n (1+C_0)/2 - J_2 + J_1\}
        \end{aligned}
    \end{equation*}
    occur, then the event in the above probability fails to occur.
    Thus, it is bounded by
    \begin{equation*}
        \begin{aligned}
            &\Prob[\mathrm{Bi}(k_n+J_2-1, F_{y \setminus x}^n) \leq k_n F_{y \setminus x}^n (1+C_0)/2]\\
            &+ \Prob[\mathrm{Po}(\lambda(n)r_n^d F_n(B_x \setminus B_y)) \geq k_n F_{y \setminus x}^n (1+C_0)/2 - J_2 + J_1].
        \end{aligned}
    \end{equation*}
    By the tail estimates for binomial and Poisson distributions (see Appendix), this is at most
    \begin{equation*}
        \exp\bigg(-k_n F_{y \setminus x}^n H\bigg(\frac{1+C_0}{2}\bigg)\bigg) + \exp\bigg(-\lambda(n) r_n^d F_n(B_x \setminus B_y) H\bigg(\frac{k_n F_{y \setminus x}^n (1+C_0)/2 - J_2 + J_1 }{\lambda(n)r_n^d F_n(B_x \setminus B_y)}\bigg)\bigg)
    \end{equation*}
    for all sufficiently large $n$.

    Similarly, if $y \notin B_x$, then it holds that
    \begin{equation*}
        h_n(x,y) \leq cn^{-1} \Prob[\CP_{\lambda(n)r_n^d}(B_x) \in A_n \mid \CP_{\lambda(n)r_n^d}(B_y) = k_n+J_2]
    \end{equation*}
    and the same upper bound for the conditional probability applies here as well.
    Now, if $x \in [-1/(2r_n),1/(2r_n)]^d$ is at a distance of at least $4$ from the boundary, then it holds that $F_{y \setminus x}^n = F_{y-x} := \mathrm{Leb}(B_{y-x} \setminus B_{\bz})/\theta$ for any $y \in B(x; 3)$.
    Furthermore, if $J_2-J_1 \leq \alpha k_n F_{y-x}$ for a sufficiently small $\alpha > 0$, it follows that
    \begin{equation*}
        \begin{aligned}
            h_n(x,y) &\leq cn^{-1} \bigg(\exp\bigg(-k_n F_{y-x} H\bigg(\frac{1+C_0}{2}\bigg)\bigg) + \exp\bigg(-\lambda(n)\theta r_n^d F_{y-x} H\bigg(\frac{1+2C_0}{3C_0}\bigg)\bigg)\bigg)\\
            &\leq cn^{-1} \exp(-c_1 F_{y-x} k_n).
        \end{aligned}
    \end{equation*}
    We now note the facts that $\inf_{z \in B(\bz; 3)} (F_z/\|z\|_2) > 0$ (see \cite[Proposition~5.16]{MR1986198}), $r_n^d = \Theta(\log n/n)$ and $k_n = \Theta(\log n)$. 
    Putting all these estimates together, we arrive at
    \begin{equation*}
        \begin{aligned}
            I_2(n) &\leq c\lambda(n)^2 r_n^{2d} n^{-1} \bigg(\int_{[-1/(2r_n),1/(2r_n)]^d} \int_{B(x; 3)} \exp(-c_1 F_{y-x} k_n) \,dydx + r_n^{1-d} + r_n^{-d}k_n^{-d}\bigg)\\
            &\leq c k_n r_n^d \bigg(r_n^{-d} \int_{B(\bz; 3)} \exp(-c_1 k_n \|z\|_2) \,dz + r_n^{1-d} + k_n^{-d}\bigg)\\
            &\leq c(k_n^{1-d} + k_n r_n)\\
            &\to 0,
        \end{aligned}
    \end{equation*}
    where $\|\cdot\|_2$ denotes the Euclidean norm.
    This completes the proof for the case where $J_2 < \infty$.
    The case where $J_2 = \infty$ is readily shown by Lemma~\ref{Lemma.07} in Appendix.
\end{proof}

Lemma~\ref{Lemma.02} provides a Poisson approximation result for the sum $\sum_{j=k_n+J_1}^{k_n+J_2} W_{j,n}'$ under the Poissonized model.
To obtain the result for the binomial model, we need to de-Poissonize this result.

\begin{proposition}\label{Proposition.01}
    Let $d \geq 2$ and $f = \mathbf{1}_{[-1/2,1/2]^d}$.
    If a sequence $(r_n)_{n \in \N}$ satisfies \eqref{equation.01}, then for any two integers $0 \leq J_1 < J_2 \leq \infty$, it holds that $\sum_{j = k_n+J_1}^{k_n+J_2} W_{j,\lambda(n)} \xrightarrow{d} \mathrm{Po}(\beta (1-C_0)\sum_{j=J_1}^{J_2} C_0^j)$.
\end{proposition}

\begin{proof}
    First, we assume that $J_2 < \infty$.
    Let $N_n^-$ and $M_n$ be independent Poisson random variables with mean $n-n^{3/4}$ and $2n^{3/4}$, respectively.
    Define two Poisson point processes
    \begin{equation*}
        \CP_n^- := \sum_{i=1}^{N_n^-} \delta_{X_i}, \quad \CP_n^+ := \sum_{i=1}^{N_n^- + M_n} \delta_{X_i}
    \end{equation*}
    and let $W_{j,n}^-, W_{j,n}^+$ denote the number of vertices with degree $j$ in $G(\CP_n^-; r_n), G(\CP_n^+; r_n)$, respectively.
    Then, with high probability, it holds that $\CP_n^- \subset \CX_n \subset \CP_n^+$ and thus we have the following inequalities:
    \begin{align}
        &\sum_{j=k_n+J_1}^\infty W_{j,n}^- \leq \sum_{j=k_n+J_1}^{n-1} W_{j,n} \leq \sum_{j=k_n+J_1}^\infty W_{j,n}^+, \label{equation.04}\\
        &\sum_{j=k_n+J_2+1}^\infty W_{j,n}^- \leq \sum_{j=k_n+J_2+1}^{n-1} W_{j,n} \leq \sum_{j=k_n+J_2+1}^\infty W_{j,n}^+. \label{equation.05}
    \end{align}
    Subtracting (\ref{equation.05}) from (\ref{equation.04}) yields
    \begin{equation*}
        \sum_{j=k_n+J_1}^{k_n+J_2} W_{j,n}^- + \sum_{j=k_n+J_2+1}^\infty (W_{j,n}^- - W_{j,n}^+) \leq \sum_{j=k_n+J_1}^{k_n+J_2} W_{j,n} \leq \sum_{j=k_n+J_1}^{k_n+J_2} W_{j,n}^+ + \sum_{j=k_n+J_2+1}^\infty (W_{j,n}^+ - W_{j,n}^-).
    \end{equation*}
    By the construction of $\CP_n^-$ and $\CP_n^+$, the left-most side of the above inequality is always bounded above by the right-most side.
    In addition, the expectation of their difference vanishes, which in particular yields the $L^1$-convergence to $0$.
    From Lemma~\ref{Lemma.02}, we have
    \begin{equation*}
        \sum_{j=k_n+J_1}^{k_n+J_2} W_{j,n}^- \xrightarrow{d} \mathrm{Po}\bigg(\beta (1-C_0) \sum_{j=J_1}^{J_2} C_0^j\bigg)
    \end{equation*}
    holds and the above observation yields the same result for $\sum_{j=k_n+J_1}^{k_n+J_2} W_{j,n}$.
    The case where $J_2 = \infty$ is simpler since we do not need (\ref{equation.05}).
\end{proof}

\begin{proof}[Proof of Theorem~\ref{Theorem.01}]
    Set $x \in \R$ and $\beta := e^{-x}$.
    Then, there exists a sequence $(\alpha_n)_{n \in \N}$ such that $\lim_{n \to \infty} \alpha_n = 0$ and
    \begin{equation*}
        \beta = \frac{1}{(1-C_0) \sqrt{2\pi k_n}} \bigg(1-\frac{x-\log ((1-C_0) \sqrt{2\pi k_n}) + \alpha_n}{k_n}\bigg)^{k_n}.
    \end{equation*}
    For simplicity, we write $x_n := x-\log ((1-C_0) \sqrt{2\pi k_n}) + \alpha_n$ and then by the Taylor expansion of $W_0$, we have
    \begin{equation*}
        \begin{aligned}
            -W_0\bigg(-\frac{1}{e} \bigg(\frac{\beta (1-C_0) \sqrt{2\pi k_n}}{n}\bigg)^{1/k_n}\bigg) &= -W_0\bigg(-\frac{1}{en^{1/k_n}} \bigg(1-\frac{x_n}{k_n}\bigg)\bigg)\\
            &= \sum_{m=1}^\infty \frac{m^{m-1}}{m!} \bigg(\frac{1}{en^{1/k_n}} \bigg(1-\frac{x_n}{k_n}\bigg)\bigg)^m.
        \end{aligned}
    \end{equation*}
    Again, by the Taylor expansion for $(1-t)^m$, there exists an $\eta \in (0,1)$ such that
    \begin{equation*}
        \bigg(1-\frac{x_n}{k_n}\bigg)^m = 1 - \frac{mx_n}{k_n} + \binom{m}{2} \bigg(1-\frac{\eta x_n}{k_n}\bigg)^{m-2} \bigg(\frac{x_n}{k_n}\bigg)^2.
    \end{equation*}
    Therefore, we have
    \begin{align*}
        &\sum_{m=1}^\infty \frac{m^{m-1}}{m!} \bigg(\frac{1}{en^{1/k_n}} \bigg(1-\frac{x_n}{k_n}\bigg)\bigg)^m\\
        &= -\sum_{m=1}^\infty \frac{(-m)^{m-1}}{m!} \bigg(-\frac{1}{en^{1/k_n}}\bigg)^m + \sum_{m=1}^\infty \frac{(-m)^{m-1}}{(m-1)!} \bigg(-\frac{1}{en^{1/k_n}}\bigg)^m \bigg(\frac{x_n}{k_n}\bigg)\\
        &\quad + O\bigg(\sum_{m=1}^\infty \frac{m^{m+1}}{m!} \bigg(\frac{1}{en^{1/k_n}}\bigg)^m \bigg(\frac{x_n}{k_n}\bigg)^2\bigg)\\
        &= -W_0\bigg(-\frac{1}{en^{1/k_n}}\bigg) - \frac{1}{en^{1/k_n}} W_0'\bigg(-\frac{1}{en^{1/k_n}}\bigg) \bigg(\frac{x_n}{k_n}\bigg) + O\bigg(\frac{x_n^2}{k_n^2}\bigg).
    \end{align*}
    If $(r_n)_{n \in \N}$ satisfies \eqref{equation.02}, then this expression equals $n\theta r_n^d/k_n$, which yields
    \begin{equation*}
        n\theta r_n^d = -W_0\bigg(-\frac{1}{en^{1/k_n}}\bigg) k_n - \frac{1}{en^{1/k_n}} W_0'\bigg(-\frac{1}{en^{1/k_n}}\bigg) x_n + O(x_n^2/k_n)
    \end{equation*}
    Therefore, we obtain
    \begin{equation*}
        \Prob[S_{k_n}(\CX_n) > r_n] = \Prob[-n\theta S_{k_n}(\CX_n)^d < -a_n k_n + b_n x_n + O(x_n^2/k_n)].
    \end{equation*}
    Recalling the definition of $x_n$, there exists a deterministic sequence $(\eta_n)_{n \in \N}$ converging to $0$ such that the probability on the right-hand side reduces to
    \begin{equation*}
        \Prob\bigg[-n\theta S_{k_n}(\CX_n)^d + a_n k_n + \frac{1}{2} b_n \log k_n + b_{\infty} \log ((1-a_{\infty}) \sqrt{2\pi}) + \eta_n < b_{\infty} x\bigg].
    \end{equation*}
    Since (\ref{equation.06}) holds and $\sum_{j \geq k_n} W_{j,n} \xrightarrow{d} \mathrm{Po}(\beta) = \mathrm{Po}(e^{-x})$ by Proposition~\ref{Proposition.01}, the above expression converges to $\exp(-\exp(-x))$.
    Further, $\eta_n$ converges in probability to $0$ and thus it follows from Slutsky's theorem that
    \begin{equation*}
        -n\theta S_{k_n}(\CX_n)^d + a_n k_n + \frac{1}{2} b_n \log k_n + b_{\infty} \log ((1-a_{\infty}) \sqrt{2\pi}) \xrightarrow{d} Z(b_{\infty}).
    \end{equation*}
    This completes the proof.
\end{proof}

Hereafter, we focus on the case where the density $f$ of $X_1$ satisfies Condition~\ref{Condition.01}.
As in the uniform case, we seek a sequence $(r_n)_{n \in \N}$ such that
\begin{equation*}
    \frac{n\theta r_n^d f_{\max}}{k_n} \to C_0 \in (0,\infty) \quad \text{and} \quad \E\bigg[\sum_{j=k_n}^\infty W_{j,n}'\bigg] \to \beta \in (0,\infty).
\end{equation*}
In this case, we have
\begin{equation*}
    \begin{aligned}
        \E\bigg[\sum_{j \geq k_n} W_{j,n}'\bigg] &= n \int_{\Rd} \sum_{j \geq k_n} \frac{(nF(B(x; r_n)))^j}{j!} e^{-nF(B(x; r_n))} f(x) \,dx\\
        &= n \frac{(n\theta r_n^d f_{\max})^{k_n}}{k_n!} e^{-n\theta r_n^d f_{\max}} \times \int_{\Rd} \bigg(\frac{F(B(x; r_n))}{\theta r_n^d f_{\max}}\bigg)^{k_n} \sum_{j \geq 0} \frac{(nF(B(x; r_n)))^j}{(k_n+j)_j}\\
        &\hspace{14em} \times \exp(n\theta r_n^d f_{\max} - nF(B(x; r_n))) f(x) \,dx.
    \end{aligned}
\end{equation*}
Using the change of variables $x = k_n^{-1/(2m)} z$ and an argument similar to that following Remark~4.6 in \cite{otsuka2026limit}, we see that the above expression is asymptotically equal to
\begin{equation*}
    \begin{aligned}
        &n \frac{(n\theta r_n^d f_{\max})^{k_n}}{k_n!} e^{-n\theta r_n^d f_{\max}} \times k_n^{-d/(2m)} \int_{\Rd} \exp\bigg(\frac{(1-C_0) \partial^{2m}f(\bz)[z^{\otimes 2m}]}{(2m)! f_{\max}}\bigg) \frac{f_{\max}}{1-C_0} \,dz\\
        &= \gamma (f_{\max}/(1-C_0))^{1+d/(2m)} \times n \frac{(n\theta r_n^d f_{\max})^{k_n}}{k_n! k_n^{d/(2m)}} e^{-n\theta r_n^d f_{\max}},
    \end{aligned}
\end{equation*}
where $\gamma := \int_{\Rd} \exp(\partial^{2m}f(\bz)[z^{\otimes 2m}]/((2m)!)) \,dz$.
Therefore, to ensure that the expectation $\E[\sum_{j \geq k_n} W_{j,n}']$ converges to $\beta$, it suffices by Stirling's formula to take
\begin{equation*}
    \frac{n\theta r_n^d f_{\max}}{k_n} = -W_0\Bigg(-\frac{1}{e}\Bigg(\frac{\beta \sqrt{2\pi k_n^{1+(d/m)}}}{n \gamma (f_{\max}/(1-C_0))^{1+d/(2m)}}\Bigg)^{1/k_n}\Bigg)
\end{equation*}
as in the uniform case.
The rest of the proof is similar to that of Theorem~\ref{Theorem.01}; hence, we outline only the necessary modifications to Lemma~\ref{Lemma.02}.
We split the integral $I_2(n)$ into two parts:
\begin{equation*}
    \begin{aligned}
        I_2(n) &= \lambda(n)^2 \int_{B(\bz; \rho_1+4r_n)} F(dx) \int_{B(x; 3r_n)} F(dy) \Prob[\{\CP_{\lambda(n)}^y(B_x^n) \in A_n\} \cap \{\CP_{\lambda(n)}^x(B_y^n) \in A_n\}]\\
        &\quad + \lambda(n)^2 \int_{\Rd \setminus B(\bz; \rho_1+4r_n)} F(dx) \int_{B(x; 3r_n)} F(dy) \Prob[\{\CP_{\lambda(n)}^y(B_x^n) \in A_n\} \cap \{\CP_{\lambda(n)}^x(B_y^n) \in A_n\}].
    \end{aligned}
\end{equation*}
Note that by the uniform continuity of $f$ on a neighborhood of $\bz$, for any $\varepsilon_1 > 0$, we can choose $\rho_0$ such that $f(x) \geq (1-\varepsilon_1)f_{\max}$ on $B(\bz; \rho_0)$.
Then, the first term is bounded above by
\begin{equation*}
    \lambda(n)^2 r_n^{2d} \int_{B(\bz; \rho_1 r_n^{-1}+4)} \int_{B(x; 3)} h_n(x,y) f_{\max}^2 \,dydx,
\end{equation*}
where we define $h_n(x,y) := \Prob[\{\CP_{\lambda(n)r_n^d}^y(B(x; 1)) \in A_n\} \cap \{\CP_{\lambda(n)r_n^d}^x(B(y; 1)) \in A_n\}]$ and let $\CP_{\lambda(n)r_n^d}$ be a Poisson point process with intensity $\lambda(n)r_n^d f(r_n x)$.
We can estimate the integrand by
\begin{equation*}
    \begin{aligned}
        &c \Prob[\CP_{\lambda(n)r_n^d}(B_x)+1 \in A_n] \times (\Prob[\mathrm{Bi}(k_n+J_2-1, (1-\varepsilon_1) F_{y-x}) \leq k_n F_{y-x} (1+C_0)/2]\\
        &\hspace{12em} + \Prob[\mathrm{Po}(\lambda(n)\theta r_n^d f_{\max} F_{y-x}) \geq k_n F_{y-x} (1+C_0)/2 - J_2+J_1]).
    \end{aligned}
\end{equation*}
If $\alpha > 0$ is sufficiently small and $J_2-J_1 \leq \alpha k_n F_{y-x}$, then this is further bounded by
\begin{equation*}
    c\exp(-c_1 k_n F_{y-x}) \Prob[\CP_{\lambda(n)r_n^d}(B_x)+1 \in A_n].
\end{equation*}
Consequently, the integral is at most $ck_n^{1-d}$, which vanishes as $n \to \infty$.
Since $f(x) \leq (1-\varepsilon_0) f_{\max}$ on $B(\bz; \rho_1)$, the second term of $I_2(n)$ is bounded by
\begin{equation*}
    c \lambda(n)^2 r_n^d \frac{(\lambda(n)r_n^d f_{\max})^{k_n}}{k_n!} e^{-\lambda(n)\theta r_n^d f_{\max}} (1-\varepsilon_0/2)^{k_n} \leq c k_n^{1+d/(2m)} \exp(-c_1 k_n)
\end{equation*}
and this vanishes as $n \to \infty$.
On the other hand, we have
\begin{equation*}
    I_1(n) \leq c \bigg(\lambda(n) \frac{(\lambda(n)\theta r_n^d f_{\max})^{k_n}}{k_n!} e^{-\lambda(n)\theta r_n^d f_{\max}}\bigg)^2 r_n^d \leq c k_n^{2+(d/m)} r_n^d \to 0,
\end{equation*}
since $r_n$ decays faster than any power of $k_n^{-1}$.
These are all the modifications needed to complete the proof.

\subsection{Point process}
First, we consider the case where the density is uniform.
To show the convergence of the point processes, we utilize the convergence criterion given by \cite[Theorem~4.15]{MR3642325}.

Let $S$ be a locally compact, second-countable Hausdorff space, and let $\CS$ be its Borel $\sigma$-field.
We write $\hat{\CS}$ for the family of relatively compact Borel sets of $S$.
A collection of Borel sets $\CK \subset \hat{\CS}$ is said to be dissecting if the following conditions hold:
\begin{itemize}
    \item Each open set $G \subset S$ is a countable union of sets in $\CK$.
    \item Each $B \in \hat{\CS}$ is covered by finitely many sets in $\CK$. 
\end{itemize}
For a point process $\xi$ on $S$, we write $\hat{\CS}_{\xi} := \{B \in \hat{\CS} : \xi(\partial B) = 0,~ \text{a.s.}\}$.
\begin{lemma}[{\cite[Theorem~4.15]{MR3642325}}]\label{Lemma.03}
    Let $\{\xi_n\}_{n \in \N}$ be a sequence of point processes on $S$, and let $\xi$ be a simple point process on $S$.
    Fix a dissecting ring $\CU \subset \hat{\CS}_{\xi}$ and a dissecting semi-ring $\CI \subset \CU$.
    Then, $\xi_n \xrightarrow{d} \xi$ if and only if the following conditions hold:
    \begin{enumerate}
        \item[(1)] For all $U \in \CU$, $\Prob[\xi_n(U) = 0] \to \Prob[\xi(U) = 0]$ as $n \to \infty$.
        \item[(2)] For all $I \in \CI$, $\limsup_{n \to \infty} \Prob[\xi_n(I) > 1] \leq \Prob[\xi(I) > 1]$.
    \end{enumerate}
\end{lemma}

To prove Theorem~\ref{Theorem.03}, we need to construct a dissecting ring and a dissecting semi-ring on $\Rd \times \Z$.
First, as a dissecting semi-ring, we define
\begin{equation*}
    \CI := \{(x,y]_{\Rd} \times [s,t]_{\Z} : x,y \in \Rd ~\text{with}~ x<y, ~\text{and}~ s,t \in \Z ~\text{with}~ s \leq t\}.
\end{equation*}
Here, the relation $x<y$ means that all coordinates of $y - x$ are positive, and we define $(x,y]_{\Rd} := \prod_{i=1}^d (x_i,y_i]_{\R}$ and $[s,t]_{\Z} := \{\ell \in \Z : s \leq \ell \leq t\}$.
Note that each element in $\CI$ is a continuity set of $\CQ$, and consequently $\CI \subset \hat{S}_{\CQ}$ for $S := \Rd \times \Z$.
One can easily check that $\CI$ indeed satisfies the conditions of being dissecting.
Since $\CI$ is a dissecting semi-ring, the collection $\CU$ of all finite disjoint unions of elements in $\CI$ forms a dissecting ring.
We use these collections, which are contained in $\hat{S}_{\CQ}$.

\begin{lemma}[{\cite[Theorem~2.3]{MR1986198}}]\label{Lemma.05}
    Let $(\CJ, \sim)$ be a finite simple graph and let $(\xi_j)_{j \in \CJ}$ be a collection of Bernoulli random variables, such that $\xi_j$ is independent of $\xi_i$ whenever $j \nsim i$.
    That is, assume that $(\CJ, \sim)$ is a dependency graph for $(\xi_j)_{j \in \CJ}$.
    Let $\{\CJ(i)\}_{i=1}^{\ell}$ be a partition of $\CJ$ and for $1 \leq i \leq \ell$, set $W_i := \sum_{j \in \CJ(i)} \xi_j$.
    Let $(Z_i)_{i=1}^{\ell}$ be a collection of independent Poisson variables with $\E[Z_i] = \E[W_i]$.
    Then, for $p_j := \E[\xi_j]$ and $p_{ji} := \E[\xi_j \xi_i]$, it holds that
    \begin{equation*}
        d_{\mathrm{TV}}\Big((W_i)_{i=1}^{\ell}, (Z_i)_{i=1}^{\ell}\Big) \leq 3\Bigg(\sum_{j \in \CJ} \sum_{i \in \CN_j \setminus \{j\}}p_{ji} + \sum_{j \in \CJ} \sum_{i \in \CN_j} p_j p_i\Bigg),
    \end{equation*}
    where $\CN_j := \{j\} \cup \{i \in \CJ : i \sim j\}$ is the neighborhood of $j$ with respect to the adjacency relation.
\end{lemma}

\begin{lemma}\label{Lemma.04}
    Assume that $f$ is bounded.
    Let $r > 0$, $\lambda > 0$, and for $\ell \in \N$, let $\{A_j\}_{j=1}^{\ell}$ be mutually disjoint subsets of $\Z$.
    Assume also that $\{D_j\}_{j=1}^{\ell} \subset \hat{\CS}$ and let $W_{A_j, \lambda, r}'(D_j)$ denote the number of vertices lying in $D_j$, whose degrees are in $A_j$ within $G(\CP_{\lambda}; r)$.
    Then, for independent Poisson random variables $(Z_j)_{j=1}^{\ell}$ with $\E[Z_j] = \E[W_{A_j, \lambda, r}'(D_j)]$, it holds that
    \begin{equation*}
        d_{\mathrm{TV}}\Big((W_{A_j, \lambda, r}'(D_j))_{j=1}^{\ell}, (Z_j)_{j=1}^{\ell}\Big) \leq 3(I_1+I_2),
    \end{equation*}
    where we define
    \begin{equation*}
        \begin{aligned}
            I_1 &:= \lambda^2 \sum_{j_1,j_2=1}^{\ell} \int_{D_{j_1}} F(dx) \int_{D_{j_2} \cap B(x; 3r)} F(dy) \Prob[\CP_{\lambda}(B(x; r)) \in A_{j_1}] \Prob[\CP_{\lambda}(B(y; r)) \in A_{j_2}],\\
            I_2 &:= \lambda^2 \sum_{j_1,j_2=1}^{\ell} \int_{D_{j_1}} F(dx) \int_{D_{j_2} \cap B(x; 3r)} F(dy) \Prob[\{\CP_{\lambda}^y(B(x; r)) \in A_{j_1}\} \cap \{\CP_{\lambda}^x(B(y; r)) \in A_{j_2}\}].
        \end{aligned}
    \end{equation*}
\end{lemma}

\begin{proof}
    We proceed by discretizing the setting to construct a dependency graph.
    After applying Lemma~\ref{Lemma.05}, we transfer the obtained upper bound back to the original continuous setting.

    Let $m \in \N$, and partition $\Rd$ into cubes of side length $m^{-1}$.
    Retain those that can cover $\bigcup_{j=1}^{\ell} D_j$, denoted by $\{H_{m,i}\}_{i=1}^q$, and let $\{a_{m,i}\}_{i=1}^q$ be the centers of these cubes.
    Now, we define
    \begin{equation*}
        \xi_{(i,j)}^m := \mathbf{1}\{\CP_{\lambda}(D_j \cap H_{m,i}) = 1\} \cap \{\CP_{\lambda}(B(a_{m,i}; r) \setminus H_{m,i}) \in A_j\}.
    \end{equation*}
    We define the adjacency relation on the indices $(i,j) \in \{1, 2, \dots, \ell\} \times \{1, 2, \dots, q\} =: \CJ_m$ by
    \begin{equation*}
        (i,j) \sim_m (i',j') \iff 0 < \|a_{m,i} - a_{m,i'}\| \leq 3r.
    \end{equation*}
    Then, $(\CJ_m, \sim_m)$ forms a dependency graph for $\{\xi_{(i,j)}^m : (i,j) \in \CJ_m\}$.
    Therefore, by Lemma~\ref{Lemma.05}, if we let $W_j^m := \sum_{i=1}^q \xi_{(i,j)}^m$ and let $(Z_j^m)_{j=1}^{\ell}$ be independent Poisson random variables such that $\mathbb{E}[Z_j^m] = \mathbb{E}[W_j^m]$, we then have
    \begin{equation*}
        \begin{aligned}
            &d_{\mathrm{TV}}\Big((W_j^m)_{j=1}^{\ell}, (Z_j^m)_{j=1}^{\ell}\Big)\\
            &\leq 3\Bigg(\sum_{(i,j) \in \CJ_m} \sum_{(i',j') \in \CN_{(i,j)}^m \setminus \{(i,j)\}} p_{(i,j),(i',j')}^m + \sum_{(i,j) \in \CJ_m} \sum_{(i',j') \in \CN_{(i,j)}^m} p_{(i,j)}^m p_{(i',j')}^m\Bigg),
        \end{aligned}
    \end{equation*}
    where $p_{(i,j)}^m := \E[\xi_{(i,j)}^m]$, $p_{(i,j),(i',j')}^m := \E[\xi_{(i,j)}^m \xi_{(i',j')}^m]$, and
    \begin{equation*}
        \CN_{(i,j)}^m := \{(i',j') \in \CJ_m : (i',j') \sim_m (i,j)\} \cup \{(i,j)\}.
    \end{equation*}
    Note that $\lim_{m \to \infty} W_j^m = W_{A_j, \lambda, r}'(D_j)$ holds almost surely.

    Now, defining
    \begin{equation*}
        w_j^m(x) := \begin{cases}
            m^d p_{(i,j)}^m & (x \in H_{m,i}), \\
            0 & (x \notin \bigcup_{i=1}^q H_{m,i}),
        \end{cases}
    \end{equation*}
    we obtain $\int_{\Rd} w_j^m(x) dx = \E[W_j^m]$.
    By the Lebesgue density theorem, it holds that
    \begin{equation*}
        \begin{aligned}
            w_j^m(x) &= \lambda m^d \int_{H_{m,i}} f(y) \mathbf{1}_{D_j}(y) \,dy \times e^{O(m^{-d})} \Prob[\CP_{\lambda}(B(a_{m,i}; r)) \in A_j](1+o(1))\\
            &\xrightarrow{m \to \infty} \lambda f(x) \mathbf{1}_{D_j}(x) \Prob[\CP_{\lambda}(B(x; r)) \in A_j], \quad \text{a.e.}
        \end{aligned}
    \end{equation*}
    Since $w_j^m$ vanishes outside of some compact set and the inequality
    \begin{equation*}
        0 \leq w_j^m(x) \leq m^d \Prob[\CP_{\lambda}(D_j \cap H_{m,i}) = 1] \leq \lambda f_{\max}
    \end{equation*}
    holds, the bounded convergence theorem yields
    \begin{equation*}
        \E[W_j^m] = \int_{\Rd} w_j^m(x) \,dx \xrightarrow{m \to \infty} \lambda \int_{D_j} \Prob[\CP_{\lambda}(B(x; r)) \in A_j] f(x) \,dx = \E[W_{A_j, \lambda, r}'(D_j)],
    \end{equation*}
    where the last equality follows from the Palm theory for Poisson processes (see Appendix).

    By a similar argument, for the functions defined by
    \begin{equation*}
        \begin{aligned}
            &u_{j,j'}^m(x,y) := m^{2d} p_{(i,j)}^m p_{(i',j')}^m \mathbf{1}\{(i',j') \in \CN_{(i,j)}^m\},\\
            &v_{j,j'}^m(x,y) := m^{2d} p_{(i,j),(i',j')}^m \mathbf{1}\{(i',j') \in \CN_{(i,j)}^m \setminus \{(i,j)\}\},
        \end{aligned}
    \end{equation*}
    with $x \in H_{m,i}$ and $y \in H_{m,i'}$, the assumption that $\{A_j\}_{j=1}^{\ell}$ are mutually disjoint implies that
    \begin{equation*}
        \begin{aligned}
            u_{j,j'}^m(x,y) &\xrightarrow{m \to \infty} \lambda^2 f(x) \mathbf{1}_{D_j}(x) f(y) \mathbf{1}_{D_{j'}}(y) \Prob[\CP_{\lambda}(B(x; r)) \in A_j] \Prob[\CP_{\lambda}(B(y; r)) \in A_{j'}] \mathbf{1}_{B(x; 3r)},\\
            v_{j,j'}^m(x,y) &\xrightarrow{m \to \infty} \lambda^2 f(x) \mathbf{1}_{D_j}(x) f(y) \mathbf{1}_{D_{j'}}(y) \Prob[\{\CP_{\lambda}^y(B(x; r)) \in A_j\} \cap \{\CP_{\lambda}^x(B(y; r)) \in A_{j'}\}] \mathbf{1}_{B(x; 3r)}
        \end{aligned}
    \end{equation*}
    hold almost surely.
    This further yields
    \begin{equation*}
        \sum_{(i,j) \in \CJ_m} \sum_{(i',j') \in \CN_{(i,j)}^m \setminus \{(i,j)\}} p_{(i,j),(i',j')}^m \xrightarrow{m \to \infty} I_2, \quad \sum_{(i,j) \in \CJ_m} \sum_{(i',j') \in \CN_{(i,j)}^m} p_{(i,j)}^m p_{(i',j')}^m \xrightarrow{m \to \infty} I_1,
    \end{equation*}
    which completes the proof.
\end{proof}

\begin{proof}[Proof of Theorem~\ref{Theorem.03}]
    By Lemma~\ref{Lemma.03}, it suffices to show that $\Phi_n(U) \xrightarrow{d} \CQ(U) = \mathrm{Po}(\Lambda(U))$ for any $U \in \CU$.
    For a given $U$, let $\pi(U) := \{j \in \Z : (\Rd \times \{j\}) \cap U \neq \emptyset\}$ denote its natural projection onto $\Z$.
    We classify the constituent sets whose union reconstructs $U$ according to the integers representing degrees.
    That is, for each $j \in \pi(U)$, we define
    \begin{equation*}
        D(j) := \bigcup \{(x, y]_{\Rd} : (x, y]_{\Rd} \times \{j\} \subset U\},
    \end{equation*}
    so that $U$ can be written as the finite union $U = \bigsqcup_{j \in \pi(U)} D(j) \times \{j\}$.

    In what follows, we prove that the Poisson approximation holds. We regard the finite set $\pi(U)$ as an ordered set by inheriting the natural order of $\Z$.
    By Lemma~\ref{Lemma.04}, for independent Poisson random variables $(Z(j))_{j \in \pi(U)}$ satisfying $\E[Z(j)] = \E[W_{k_n+j, n}'(D(j))]$, we obtain the total variation bound
    \begin{equation*}
        d_{\mathrm{TV}}\Big((W_{k_n+j, n}'(D(j)))_{j \in \pi(U)}, \, (Z(j))_{j \in \pi(U)}\Big) \leq 3(I_1(n) + I_2(n)),
    \end{equation*}
    where $I_1(n)$ and $I_2(n)$ are defined by
    \begin{align*}
        I_1(n) &:= \lambda(n)^2 \sum_{j,\ell \in \pi(U)} \int_{D(j)} F(dx) \int_{D(\ell) \cap B(x; 3r_n)} F(dy)\\
        &\hspace{14em} \times \Prob[\CP_{\lambda(n)}(B_x^n) = k_n + j] \Prob[\CP_{\lambda(n)}(B_y^n) = k_n + \ell],\\
        I_2(n) &:= \lambda(n)^2 \sum_{j,\ell \in \pi(U)} \int_{D(j)} F(dx) \int_{D(\ell) \cap B(x; 3r_n)} F(dy)\\
        &\hspace{14em} \times \Prob[\{\CP_{\lambda(n)}^y(B_x^n) = k_n + j\} \cap \{\CP_{\lambda(n)}^x(B_y^n) = k_n + \ell\}].
    \end{align*}
    Note that these terms are bounded by $|\pi(U)|^2$ times the upper bounds appearing in Lemma~\ref{Lemma.02}, which were shown to vanish.
    Consequently, the Poisson approximation holds. 

    To complete the argument, it is necessary to de-Poissonize this result; however, this step can be carried out in the almost same manner as in the proof of Proposition~\ref{Proposition.01}.
    Through this procedure, we find that the joint distribution of $(W_{k_n+j, n}(D(j)))_{j \in \pi(U)}$ converges weakly to that of $(Z(j))_{j \in \pi(U)}$.
    Since the random variables $(Z(j))_{j \in \pi(U)}$ are independent, their sum is also Poisson distributed, yielding
    \begin{equation*}
        \Phi_n(U) = \sum_{j \in \pi(U)} W_{k_n+j, n}(D(j)) \xrightarrow{d} \sum_{j \in \pi(U)} Z(j) = \mathrm{Po}(\Lambda(U)).
    \end{equation*}
    This completes the proof for the uniform case.

    The proof for the non-uniform case proceeds in almost the same manner.
    Since only the calculation of the expectation of $W'$ differs, the interested reader is referred to the proof of Theorem~2.11 in \cite{otsuka2026limit}.
\end{proof}

\section{Proofs for the super-connectivity regime}
\subsection{Threshold radius}
As in the connectivity regime, we begin with the uniform case.
Based on the almost sure limit result in Penrose~\cite[Theorem~6.11]{MR1986198}, we search for $(r_n)_{n \in \N}$ that satisfies the following conditions:
\begin{equation}\label{equation.03}
    \frac{n\theta r_n^d}{k_n} = 1-\varepsilon_n, \quad (0,\infty) \ni \varepsilon_n \to 0 \quad \text{and} \quad \E\bigg[\sum_{j=k_n}^\infty W_{j,n}'\bigg] \to \beta \in (0,\infty).
\end{equation}
Let us evaluate the asymptotic behavior of the expectation.
Recall that $(m+\ell)_{\ell} := \prod_{i=1}^{\ell} (m+i)$ and then we have
\begin{equation*}
    \E\bigg[\sum_{j=k_n}^\infty W_{j,n}'\bigg] \sim n \sum_{j=k_n}^\infty \frac{(n\theta r_n^d)^j}{j!} e^{-n\theta r_n^d} \sim \frac{n}{\sqrt{2\pi k_n}} \bigg(\frac{en\theta r_n^d}{k_n}\bigg)^{k_n} \sum_{j=0}^\infty \frac{(n\theta r_n^d)^j}{(k_n+j)_j} e^{-n\theta r_n^d}.
\end{equation*}
First, we observe that the upper bound is given by
\begin{equation*}
    \sum_{j \geq 0} \frac{(n\theta r_n^d)^j}{(k_n+j)_j} \leq \sum_{j \geq 0} \bigg(\frac{n\theta r_n^d}{k_n}\bigg)^j = \sum_{j \geq 0} (1-\varepsilon_n)^j = \frac{1}{\varepsilon_n}.
\end{equation*}
On the other hand, for a sequence $(\alpha_n)_{n \in \N} \subset \N$, we have the lower bound
\begin{equation*}
    \sum_{j \geq 0} \frac{(n\theta r_n^d)^j}{(k_n+j)_j} \geq \sum_{j=0}^{\alpha_n} \bigg(\frac{n\theta r_n^d}{k_n+\alpha_n}\bigg)^j \geq \frac{1-(1-\varepsilon_n)^{\alpha_n+1}}{1-(1-\varepsilon_n)(k_n/(k_n+\alpha_n))}
\end{equation*}
If we can find a sequence $(\alpha_n)_{n \in \N}$ such that $\varepsilon_n^{-1} \ll \alpha_n \ll \varepsilon_n k_n$, it then follows that the lower bound is asymptotically equivalent to $\varepsilon_n^{-1}$, and so is the original sum.

We now define $\varepsilon_n$ implicitly by
\begin{equation*}
    k_n (-\varepsilon_n-\log (1-\varepsilon_n)) = \log n - \frac{1}{2} \log \log n - \log (2\sqrt{\pi} \beta)
\end{equation*}
and show that this choice of $\varepsilon_n$ satisfies (\ref{equation.03}).
Indeed, $g(x) := x \sqrt{-x^{-1} - x^{-2} \log (1-x)}$ is a smooth function with a non-vanishing derivative at the origin and
\begin{equation*}
    g(\varepsilon_n) = \sqrt{\frac{\log n - \frac{1}{2} \log \log n - \log (2\sqrt{\pi} \beta)}{k_n}}
\end{equation*}
holds by the definition of $\varepsilon_n$.
Since the Taylor expansion of the inverse function $g^{-1}$ near the origin is given by
\begin{equation}\label{equation.07}
    g^{-1}(x) = \sqrt{2}x - \frac{2}{3}x^2 + O(x^3),
\end{equation}
we see that $\varepsilon_n = \sqrt{2\log n/k_n}(1+o(1))$ and there exists $(\alpha_n)_{n \in \N}$ such that $\varepsilon_n^{-1} \ll \alpha_n \ll \varepsilon_n k_n$.
Therefore, we have
\begin{equation*}
    n \sum_{j \geq k_n} \frac{(n\theta r_n^d)^j}{j!} e^{-n\theta r_n^d} \sim \frac{n}{\sqrt{2\pi k_n}} e^{k_n} (1-\varepsilon_n)^{k_n} \frac{1}{\varepsilon_n} e^{-(1-\varepsilon_n) k_n} = \frac{n}{\sqrt{2\pi k_n}} e^{\varepsilon_n k_n} (1-\varepsilon_n)^{k_n} \frac{1}{\varepsilon_n}.
\end{equation*}
The above expression converges to $\beta$ if and only if
\begin{equation}\label{equation.08}
    \log n - \log \sqrt{2\pi k_n} + \varepsilon_n k_n + k_n \log (1-\varepsilon_n) - \log \varepsilon_n \to \log \beta
\end{equation}
holds, and (\ref{equation.08}) is immediately verified by the definition of $\varepsilon_n$.
Thus, we obtain (\ref{equation.03}) for this choice of $\varepsilon_n$.

\begin{lemma}\label{Lemma.06}
    Suppose that $d \geq 2$ and $(\lambda(n))_{n \in \N}$ satisfies $n-n^{3/4} \leq \lambda(n) \leq n+n^{3/4}$ for each $n \in \N$.
    Let $f$ be the density function $\mathbf{1}_{[-1/2,1/2]^d}$.
    If a sequence $(r_n)_{n \in \N}$ satisfies \eqref{equation.03} for $(k_n)_{n \in \N}$ such that $\log n \ll k_n \ll (\log n)^d$, then for any two real numbers $0 \leq s_1 < s_2 \leq \infty$, it holds that $\sum_{k_n+s_1\varepsilon_n^{-1} \leq j < k_n+s_2\varepsilon_n^{-1}} W_{j,\lambda(n)}' \xrightarrow{d} \mathrm{Po}(\beta(e^{-s_1} - e^{-s_2}))$.
\end{lemma}

\begin{proof}
    Note that for $s \in \R$, it holds that
    \begin{equation*}
        \E\bigg[\sum_{k_n \leq j < k_n+s\varepsilon_n^{-1}} W_{j,\lambda(n)}'\bigg] = \lambda(n) \frac{(\lambda(n)\theta r_n^d)^{k_n}}{k_n!} \sum_{0 \leq j < s\varepsilon_n^{-1}} \frac{(\lambda(n)\theta r_n^d)^j}{(k_n+j)_j} e^{-\lambda(n)\theta r_n^d}
    \end{equation*}
    The asymptotic behavior of the summation is given by
    \begin{equation*}
        \sum_{0 \leq j < s\varepsilon_n^{-1}} \frac{(\lambda(n)\theta r_n^d)^j}{(k_n+j)_j} \sim \frac{1-(1-\varepsilon_n)^{s\varepsilon_n^{-1}}}{1-(1-\varepsilon_n)} \sim \frac{1-e^{-s}}{\varepsilon_n}.
    \end{equation*}
    Since the sequence $(r_n)_{n \in \N}$ is chosen to satisfy \eqref{equation.03}, we consequently obtain
    \begin{equation*}
        \E\bigg[\sum_{k_n \leq j < k_n+s\varepsilon_n^{-1}} W_{j,\lambda(n)}'\bigg] \to \beta (1-e^{-s}), \quad \E\bigg[\sum_{k_n+s_1\varepsilon_n^{-1} \leq j < k_n+s_2\varepsilon_n^{-1}} W_{j,\lambda(n)}'\bigg] \to \beta (e^{-s_1} - e^{-s_2}).
    \end{equation*}
    As in the connectivity regime, let $A_n := \{\ell \in \N : k_n+s_1\varepsilon_n^{-1} \leq \ell < k_n+s_2\varepsilon_n^{-1}\}$ and define
    \begin{equation*}
        \begin{aligned}
            &I_1(n) := \lambda(n)^2 \int_{\Rd} F(dx) \int_{B(x; 3r_n)} F(dy) \Prob[\CP_{\lambda(n)}(B_x^n) \in A_n] \Prob[\CP_{\lambda(n)}(B_y^n) \in A_n],\\
            &I_2(n) := \lambda(n)^2 \int_{\Rd} F(dx) \int_{B(x; 3r_n)} F(dy) \Prob[\{\CP_{\lambda(n)}^y(B_x^n) \in A_n\} \cap \{\CP_{\lambda(n)}^x(B_y^n) \in A_n\}].
        \end{aligned}
    \end{equation*}
    If we can verify that these integrals vanish as $n \to \infty$, then the proof for the case where $s_2 < \infty$ is complete.

    First, we assume that $s_2 < \infty$.
    By the definition of $I_1(n)$, we obtain
    \begin{equation*}
        I_1(n) \leq \bigg(\lambda(n) \sum_{j \in A_n} \frac{(\lambda(n)\theta r_n^d)^j}{j!} e^{-\lambda(n)\theta r_n^d}\bigg)^2 3^d r_n^d \leq cr_n^d \to 0.
    \end{equation*}
    In what follows, we estimate $I_2(n)$.
    In a manner similar to that in the connectivity regime, we see that
    \begin{equation*}
        \begin{aligned}
            &I_2(n) \leq \lambda(n)^2 r_n^{2d} \int_{[-1/(2r_n), 1/(2r_n)]^d} \int_{B(x; 3)} h_n(x,y) \,dydx,\\
            &h_n(x,y) := \Prob[\{\CP_{\lambda(n)r_n^d}^y(B_x) \in A_n\} \cap \{\CP_{\lambda(n)r_n^d}^x(B_y) \in A_n\}],
        \end{aligned}
    \end{equation*}
    where $\CP_{\lambda(n)r_n^d}$ denotes a Poisson point process with intensity $\lambda(n)r_n^d \mathbf{1}_{[-1/(2r_n),1/(2r_n)]^d}$.
    Then,
    \begin{equation*}
        h_n(x,y) \leq cn^{-1} \Prob[\CP_{\lambda(n)r_n^d}(B_x) \in A_n \mid \CP_{\lambda(n)r_n^d}(B_y) = \lfloor k_n+s_2\varepsilon_n^{-1} \rfloor]
    \end{equation*}
    holds.
    The conditional probability equals
    \begin{equation*}
        \Prob[\lfloor k_n+s_2\varepsilon_n^{-1} \rfloor - \mathrm{Bi}(\lfloor k_n+s_2\varepsilon_n^{-1} \rfloor, F_{y \setminus x}^n) + \mathrm{Po}(\lambda(n)r_n^d F_n(B_x \setminus B_y)) \in A_n],
    \end{equation*}
    which is bounded above by
    \begin{equation*}
        \begin{aligned}
            &\Prob[\mathrm{Bi}(\lfloor k_n+s_2\varepsilon_n^{-1} \rfloor, F_{y \setminus x}^n) \leq (1-\varepsilon_n/2) k_n F_{y \setminus x}]\\
            &+ \Prob[\mathrm{Po}(\lambda(n)r_n^d F_n(B_x \setminus B_y)) \geq (1-\varepsilon_n/2) k_n F_{y \setminus x}^n - (s_2-s_1)\varepsilon_n^{-1}].
        \end{aligned}
    \end{equation*}
    If $x \in [-1/(2r_n), 1/(2r_n)]^d$ is at a distance of at least $4$ from the boundary, then $F_{y \setminus x}^n = F_{y-x}$ holds for any $y \in B(x; 3)$.
    Moreover, under the condition that $(s_2-s_1)\varepsilon_n^{-1} \leq \alpha \varepsilon_n k_n F_{y-x}$ for a sufficiently small $\alpha > 0$, we obtain
    \begin{equation*}
        \begin{aligned}
            h_n(x,y) &\leq cn^{-1} \bigg(\exp\bigg(-c_1 k_n F_{y-x} H\bigg(1-\frac{\varepsilon_n}{4}\bigg)\bigg) + \exp\bigg(-c_1 k_n F_{y-x} H\bigg(1+\frac{\varepsilon_n}{4}\bigg)\bigg)\bigg)\\
            &\leq cn^{-1} \exp(-c_1 \varepsilon_n^2 k_n F_{y-x})\\
            &\leq cn^{-1} \exp(-c_1 F_{y-x} \log n)
        \end{aligned}
    \end{equation*}
    by the tail estimates (see Appendix).
    Putting all these estimates together, we have
    \begin{equation*}
        \begin{aligned}
            I_2(n) &\leq c n^2 r_n^{2d} n^{-1} \bigg(\int_{[-1/(2r_n), 1/(2r_n)]^d} \int_{B(x; 3)} \exp(-c_1 F_{y-x} \log n) \,dydx + r_n^{1-d} + r_n^{-d}(\log n)^{-d}\bigg)\\
            &\leq c k_n r_n^d \bigg(r_n^{-d} \int_{B(\bz; 3)} \exp(-c_1 \|z\|_2 \log n) \,dz + r_n^{1-d} + r_n^{-d}(\log n)^{-d}\bigg)\\
            &\leq c(k_n (\log n)^{-d} + k_n r_n)\\
            &\to 0,
        \end{aligned}
    \end{equation*}
    where $\|\cdot\|_2$ denotes the Euclidean norm, and we used $r_n^d = \Theta(k_n/n)$, $k_n = o((\log n)^d)$ and the fact that $\inf_{z \in B(\bz; 3)} (F_z/\|z\|_2) > 0$ (see \cite[Proposition~5.16]{MR1986198}).
    This completes the proof for the case where $s_2 < \infty$.
    The case where $s_2 = \infty$ is readily shown by truncation (see Appendix).
\end{proof}

\begin{remark}
    We rely on the assumption $k_n = o((\log n)^d)$ to guarantee the vanishing of $I_2(n)$.
    When $k_n \gtrsim (\log n)^d$, this estimate is no longer valid, which introduces the potential for high-degree vertices to cluster around one another. In such cases, a Poisson approximation might not be applicable.
    Investigating this regime remains a subject for future research.
\end{remark}

Since the subsequent de-Poissonization follows by exactly the same argument as before, we omit the proof.

\begin{proposition}\label{Proposition.02}
    Let $d \geq 2$ and $f = \mathbf{1}_{[-1/2,1/2]^d}$.
    If a sequence $(r_n)_{n \in \N}$ satisfies \eqref{equation.03} for $(k_n)_{n \in \N}$ such that $\log n \ll k_n \ll (\log n)^d$, then for any two real numbers $0 \leq s_1 < s_2 \leq \infty$, it holds that $\sum_{k_n+s_1\varepsilon_n^{-1} \leq j < k_n+s_2\varepsilon_n^{-1}} W_{j,n} \xrightarrow{d} \mathrm{Po}(\beta(e^{-s_1} - e^{-s_2}))$.
\end{proposition}

\begin{proof}[Proof of Theorem~\ref{Theorem.04}]
    Let $x \in \R$ and $\beta := e^{-x}$.
    We define $\varepsilon_n'$ implicitly by
    \begin{equation*}
        k_n(-\varepsilon_n'-\log (1-\varepsilon_n')) = \log n - \frac{1}{2} \log \log n.
    \end{equation*}
    Note that $\varepsilon_n'$ does not depend on $\beta$.
    Moreover, it follows from (\ref{equation.07}) that
    \begin{equation*}
        \sqrt{\frac{2\log n}{k_n}} k_n  \varepsilon_n' - \log (2\sqrt{\pi} \beta) = \sqrt{\frac{2\log n}{k_n}} k_n \varepsilon_n + o(1).
    \end{equation*}
    By the definition of $\varepsilon_n$, we have
    \begin{equation*}
        \begin{aligned}
            \Prob[S_{k_n}(\CX_n) > r_n] &= \Prob[-n\theta S_{k_n}(\CX_n)^d < -n\theta r_n^d]\\
            &= \Prob[-n\theta S_{k_n}(\CX_n)^d < -(1-\varepsilon_n) k_n],
        \end{aligned}
    \end{equation*}
    which coincides with
    \begin{equation*}
        \Prob\bigg[\sqrt{\frac{2\log n}{k_n}}(-n\theta S_{k_n}(\CX_n)^d + k_n - \varepsilon_n k_n) < 0\bigg].
    \end{equation*}
    Here, from the relation between $\varepsilon_n$ and $\varepsilon_n'$, the above probability can be rewritten as
    \begin{equation*}
        \Prob\bigg[\sqrt{\frac{2\log n}{k_n}} \bigg(-n\theta S_{k_n}(\CX_n)^d + k_n - \varepsilon_n' k_n + \sqrt{\frac{k_n}{2\log n}} \log (2\sqrt{\pi})\bigg) + o(1) < x \bigg].
    \end{equation*}
    It follows from Proposition~\ref{Proposition.02} that this expression converges to $\exp(-\exp(-x))$, which implies the convergence in distribution to a Gumbel random variable.
    The asymptotic behavior of $(a_n)_{n \in \N}$ is an immediate consequence of the Taylor expansion (\ref{equation.07}) of $g^{-1}$.
\end{proof}

Next, we turn our attention to the case of a non-uniform density.
Similar to the uniform case, we seek a sequence $(r_n)_{n \in \N}$ such that
\begin{equation*}
    \frac{n\theta r_n^d f_{\max}}{k_n} = 1-\varepsilon_n, \quad (0,\infty) \ni \varepsilon_n \to 0 \quad \text{and} \quad \E\bigg[\sum_{j=k_n}^\infty W_{j,n}'\bigg] \to \beta \in (0,\infty).
\end{equation*}
As in the connectivity regime, we obtain
\begin{equation*}
    \begin{aligned}
        \E\bigg[\sum_{j \geq k_n} W_{j,n}'\bigg] &= n \frac{(n\theta r_n^d f_{\max})^{k_n}}{k_n!} e^{-n\theta r_n^d f_{\max}} \times \int_{\Rd} \bigg(\frac{F(B(x; r_n))}{\theta r_n^d f_{\max}}\bigg)^{k_n} \sum_{j \geq 0} \frac{(nF(B(x; r_n)))^j}{(k_n+j)_j}\\
        &\hspace{14em} \times \exp(n\theta r_n^d f_{\max} - nF(B(x; r_n))) f(x) \,dx.
    \end{aligned}
\end{equation*}
Let us now estimate the integral.
Let $\rho_0 > \rho_1 > 0$ and $\varepsilon_0 > 0$ be given as in Condition~\ref{Condition.01}.
Since $t^{k_n} e^{-t}$ is increasing on $[0,k_n]$ and $n\theta r_n^d f_{\max} \leq k_n$, it follows that
\begin{equation*}
    \begin{aligned}
        &\int_{\Rd \setminus B(\bz; \rho_1+r_n)} \bigg(\frac{F(B(x; r_n))}{\theta r_n^d f_{\max}}\bigg)^{k_n} \sum_{j \geq 0} \frac{(nF(B(x; r_n)))^j}{(k_n+j)_j} \exp(n\theta r_n^d f_{\max} - nF(B(x; r_n))) f(x) \,dx\\
        &\leq \int_{\Rd \setminus B(\bz; \rho_1+r_n)} (1-\varepsilon_0)^{k_n} \frac{1}{\varepsilon_n} \exp(\varepsilon_0 n\theta r_n^d f_{\max}) f(x) \,dx\\
        &\leq \frac{1}{\varepsilon_n} \exp(-c \varepsilon_n k_n).
    \end{aligned}
\end{equation*}

On the other hand, since $\rho_1+r_n \leq (\rho_0+\rho_1)/2$ for all sufficiently large $n$ and $f$ is Lipschitz continuous on $B(\bz; (\rho_0+\rho_1)/2)$, we have
\begin{equation*}
    \begin{aligned}
        &\int_{B(\bz; \rho_1+r_n)} \bigg(\frac{F(B(x; r_n))}{\theta r_n^d f_{\max}}\bigg)^{k_n} \sum_{j \geq 0} \frac{(nF(B(x; r_n)))^j}{(k_n+j)_j} \exp(n\theta r_n^d f_{\max} - nF(B(x; r_n))) f(x) \,dx\\
        &= \int_{B(\bz; \rho_1+r_n)} \bigg(\frac{f(x)}{f_{\max}}\bigg)^{k_n} (1+O(r_n))^{k_n} \sum_{j \geq 0} \frac{(n\theta r_n^d f(x))^j (1+O(r_n))^j}{(k_n+j)_j}\\
        &\quad \times \exp(n\theta r_n^d f_{\max} - n\theta r_n^d f(x)(1+O(r_n))) f(x) \,dx\\
        &\sim \int_{B(\bz; \rho_1+r_n)} \bigg(\frac{f(x)}{f_{\max}}\bigg)^{k_n} \sum_{j \geq 0} \frac{(n\theta r_n^d f(x))^j}{(k_n+j)_j} \exp(n\theta r_n^d f_{\max} - n\theta r_n^d f(x)) f(x) \,dx,
    \end{aligned}
\end{equation*}
where we used the facts that $k_n r_n \to 0$ and $n\theta r_n^d \times r_n \to 0$, which follow from the assumptions $n\theta r_n^d = k_n (1+o(1))$ and $k_n = o((\log n)^d)$.
Now, we define $\varepsilon_n$ implicitly by
\begin{equation*}
    k_n(-\varepsilon_n - \log (1-\varepsilon_n)) = \log n - \frac{d}{4m} \log k_n - \frac{1}{2}\bigg(1+\frac{d}{2m}\bigg)\log \log n - \log \frac{2^{1+d/(4m)} \sqrt{\pi} \beta}{\gamma f_{\max}^{1+d/(2m)}}.
\end{equation*}
Since $\log k_n = \Theta(\log \log n)$ by the assumption $\log n \ll k_n \ll (\log n)^d$, it holds from the Taylor expansion (\ref{equation.07}) of $g^{-1}$ that the first order of $\varepsilon_n$ is given by $\sqrt{2\log n/k_n}$.
By the Taylor expansion of $f$, if we change the variable to $x = (\varepsilon_n/\log n)^{1/(2m)} z = \widetilde{\varepsilon}_n z$, then the above integrand is asymptotic to
\begin{align*}
    &\bigg(1+\frac{\varepsilon_n \partial^{2m}f(\xi_z \widetilde{\varepsilon}_n z)[z^{\otimes 2m}]}{(2m)! f_{\max} \log n}\bigg)^{k_n} \frac{1}{\varepsilon_n} \frac{1}{1-\partial^{2m}f(\xi_z \widetilde{\varepsilon}_n z)[z^{\otimes 2m}]/((2m)! f_{\max} \log n)}\\
    &\quad \times \exp\bigg(\varepsilon_n n\theta r_n^d f_{\max} \frac{\partial^{2m}f(\xi_z \widetilde{\varepsilon}_n z)[z^{\otimes 2m}]}{(2m)! f_{\max} \log n}\bigg) f(\widetilde{\varepsilon}_n z)\\
    &\sim \frac{1}{\varepsilon_n} \exp\bigg(\frac{\varepsilon_n^2 k_n}{\log n} \times \frac{\partial^{2m}f(\xi_z \widetilde{\varepsilon}_n z)[z^{\otimes 2m}]}{(2m)! f_{\max}}\bigg) f_{\max}\\
    &\sim \frac{1}{\varepsilon_n} \exp\bigg(\frac{\partial^{2m}f(\bz)[z^{\otimes 2m}]}{(2m)! f_{\max}/2}\bigg) f_{\max}.
\end{align*}
Since the integrand is dominated by $\varepsilon_n^{-1} \exp(\partial^{2m}f(\bz)[z^{\otimes 2m}]/((2m)! f_{\max})) f_{\max}$ for all sufficiently large $n$, the dominated convergence theorem implies that
\begin{equation*}
    \begin{aligned}
        &\int_{B(\bz; \rho_1+r_n)} \bigg(\frac{f(x)}{f_{\max}}\bigg)^{k_n} \sum_{j \geq 0} \frac{(n\theta r_n^d f(x))^j}{(k_n+j)_j} \exp(n\theta r_n^d f_{\max} - n\theta r_n^d f(x)) f(x) \,dx\\
        &\sim \frac{1}{\varepsilon_n} \bigg(\frac{\varepsilon_n}{\log n}\bigg)^{d/(2m)} \int_{\Rd} \exp\bigg(\frac{\partial^{2m}f(\bz)[z^{\otimes 2m}]}{(2m)! f_{\max}/2}\bigg) f_{\max}\,dz\\
        &= \frac{1}{\varepsilon_n} \bigg(\frac{\varepsilon_n}{2\log n}\bigg)^{d/(2m)} \gamma f_{\max}^{1+d/(2m)}.
    \end{aligned}
\end{equation*}
Hence, we have
\begin{equation*}
    \begin{aligned}
        \E\bigg[\sum_{j \geq k_n} W_{j,n}'\bigg] &\sim \gamma f_{\max}^{1+d/(2m)} n \frac{(n\theta r_n^d f_{\max})^{k_n}}{k_n!} \frac{1}{\varepsilon_n} \bigg(\frac{\varepsilon_n}{2\log n}\bigg)^{d/(2m)} e^{-n\theta r_n^d f_{\max}}\\
        &\sim \gamma f_{\max}^{1+d/(2m)} \frac{n}{\sqrt{2\pi k_n}} e^{\varepsilon_n k_n} (1-\varepsilon_n)^{k_n} \frac{1}{\varepsilon_n} \bigg(\frac{\varepsilon_n}{2\log n}\bigg)^{d/(2m)}.
    \end{aligned}
\end{equation*}
The convergence of this expression to $\beta$ is equivalent to the condition that
\begin{equation*}
    \begin{aligned}
        &\log n - \log \sqrt{k_n} + \varepsilon_n k_n + k_n \log (1-\varepsilon_n) + \bigg(\frac{d}{2m} - 1\bigg)\log \varepsilon_n - \frac{d}{2m}\log \log n + \log \frac{\gamma f_{\max}^{1+d/(2m)}}{\sqrt{2\pi} 2^{d/(2m)}}\\
        &= \log n - \frac{d}{4m} \log k_n - \frac{1}{2}\bigg(1+\frac{d}{2m}\bigg)\log \log n + \log \frac{\gamma f_{\max}^{1+d/(2m)}}{\sqrt{\pi} 2^{1+d/(4m)}} + \varepsilon_n k_n + k_n \log (1-\varepsilon_n) + o(1)\\
        &\to \log \beta,
    \end{aligned}
\end{equation*}
which is immediately verified by the definition of $\varepsilon_n$.

If we define $\varepsilon_n'$ implicitly by
\begin{equation*}
    k_n(-\varepsilon_n' - \log (1-\varepsilon_n')) = \log n - \frac{d}{4m} \log k_n - \frac{1}{2}\bigg(1+\frac{d}{2m}\bigg)\log \log n,
\end{equation*}
then it follows that
\begin{equation*}
    \sqrt{\frac{2\log n}{k_n}} k_n \varepsilon_n' - \log \frac{2^{1+d/(4m)} \sqrt{\pi} \beta}{\gamma f_{\max}^{1+d/(2m)}} = \sqrt{\frac{2\log n}{k_n}} k_n \varepsilon_n + o(1).
\end{equation*}
Therefore, in a manner nearly identical to the uniform case, we can establish the Poisson approximation for $\sum_{j=k_n}^{n-1} W_{j,n}$, which yields the following convergence in distribution:
\begin{equation*}
    \frac{-n\theta S_{k_n}(\CX_n)^d + a_n}{b_n} \xrightarrow{d} Z,
\end{equation*}
where $Z$ is a standard Gumbel random variable, $b_n := \sqrt{k_n/(2\log n)}$ and the sequence $(a_n)_{n \in \N}$ satisfies
\begin{equation*}
    a_n = k_n - \sqrt{2k_n \log n} + \bigg(\frac{2}{3}\log n + \sqrt{\frac{k_n}{8\log n}} \bigg\{\bigg(1+\frac{d}{2m}\bigg)\log \log n + \frac{d}{4m} \log k_n\bigg\}\bigg) (1+o(1)).
\end{equation*}

\subsection{Point process}
Although the expectation is calculated in a similar manner to the previous section, we must now properly choose $k_n$ depending on the given value of $r_n$ so that the expectation of $\sum_{j=k_n}^\infty W_{j,n}'$ converges to some constant (and such a choice is possible only in the super-connectivity regime).
For simplicity, throughout this section only, we write $\mu_n := n\theta r_n^d$.

Analogously to the previous section, we assume that $k_n = (1+\varepsilon_n) \mu_n$ with $(0,\infty) \ni \varepsilon_n \to 0$ as $n \to \infty$.
If we define $\varepsilon_n$ implicitly by
\begin{equation*}
    \mu_n(-\varepsilon_n + (1+\varepsilon_n) \log (1+\varepsilon_n)) = \mu_n \int_0^{\varepsilon_n} \log (1+t) \,dt = \log n - \frac{1}{2} \log \log n - \log (2\sqrt{\pi}),
\end{equation*}
then it holds that $\varepsilon_n = \sqrt{2\log n/\mu_n} (1+o(1))$.
This leading term of $\varepsilon_n$ can be deduced from the Taylor expansion of the inverse of $\widetilde{g}(x) := x\sqrt{-x^{-1}+ x^{-2}(1+x) \log (1+x)}$ given by
\begin{equation*}
    \widetilde{g}^{-1}(x) = \sqrt{2}x + \frac{1}{3}x^2 + O(x^3).
\end{equation*}
Then, we obtain the asymptotic behavior of the expectation:
\begin{equation*}
    \E\bigg[\sum_{j=k_n}^\infty W_{j,n}'\bigg] \sim \frac{n}{\sqrt{2\pi k_n}} e^{k_n} (1+\varepsilon_n)^{-k_n} \frac{1}{\varepsilon_n} e^{-\mu_n} \sim \frac{n}{\sqrt{2\pi \mu_n}} e^{\varepsilon_n \mu_n} (1+\varepsilon_n)^{-(1+\varepsilon_n)\mu_n} \frac{1}{\varepsilon_n}.  
\end{equation*}
Our goal now is to ensure that this converges to $1$, which is equivalent to the condition that
\begin{equation*}
    \log n - \log \sqrt{2\pi \mu_n} + \varepsilon_n \mu_n - \mu_n (1+\varepsilon_n) \log (1+\varepsilon_n) -\log \varepsilon_n \to 0.
\end{equation*}
This is directly verified by the definition of $\varepsilon_n$.

To prove Theorems~\ref{Theorem.06} and \ref{Theorem.07}, we introduce a dissecting ring $\CU$ and a dissecting semi-ring $\CI$ in $\Rd \times \R = \R^{d+1}$.
We define the dissecting semi-ring $\CI$ as
\begin{equation*}
    \CI := \{(x,y]_{\R^{d+1}} : x,y \in \R^{d+1} ~\text{with}~ x < y\}
\end{equation*}
and the dissecting ring $\CU$ as the collection of all finite disjoint unions of elements of $\CI$.
Then, to complete the proofs of these theorems, it suffices to verify that
\begin{equation*}
    \Phi_n(U) \xrightarrow{d} \CQ(U) = \mathrm{Po}(\Lambda(U))
\end{equation*}
holds for any $U \in \CU$, by an argument almost identical to the proof of Theorem~\ref{Theorem.03}.

\section{Proofs of corollaries}
\begin{proof}[Proof of Corollary~\ref{Corollary.01}]
    Fix $\varepsilon > 0$.
    In the uniform and connectivity regime case, we have
    \begin{equation*}
        \begin{aligned}
            &\Prob\bigg[\bigg|\frac{n\theta S_{k_n}(\CX_n)^d}{k_n} - \frac{-W_0(-e^{-1-1/b})}{f_{\max}}\bigg| \geq \varepsilon\bigg]\\
            &\leq \Prob\bigg[\bigg|-n\theta S_{k_n}(\CX_n)^d + a_n k_n + \frac{1}{2} b_n \log k_n + b_{\infty} \log ((1-a_{\infty}) \sqrt{2\pi})\bigg| \geq k_n \varepsilon/2\bigg]\\
            &\leq \Prob\bigg[\bigg|-n\theta S_{k_n}(\CX_n)^d + a_n k_n + \frac{1}{2} b_n \log k_n + b_{\infty} \log ((1-a_{\infty}) \sqrt{2\pi})\bigg| \geq M\bigg]
        \end{aligned}
    \end{equation*}
    for all sufficiently large $n$ and $0 < M \leq k_n \varepsilon/2$.
    Theorem~\ref{Theorem.01} and Portmanteau theorem yield
    \begin{equation*}
        \limsup_{n \to \infty} \Prob\bigg[\bigg|\frac{n\theta S_{k_n}(\CX_n)^d}{k_n} - \frac{-W_0(-e^{-1-1/b})}{f_{\max}}\bigg| \geq \varepsilon\bigg] \leq \Prob[|Z_b| \geq M].
    \end{equation*}
    Since this holds for any $M > 0$, the left-hand side must be zero.
    This completes the proof for the case.
    The proofs for other cases follow almost identical lines.
\end{proof}

\begin{proof}[Proof of Corollary~\ref{Corollary.02}]
    In the proof of \cite[Theorem~6.14]{MR1986198}, the almost sure limit of the maximum degree $\Delta_n$ is deduced from that of the threshold radius $S_{k_n}(\CX_n)$.
    By simply rephrasing this approach in terms of convergence in probability, the desired result readily follows.
\end{proof}

\begin{proof}[Proof of Corollary~\ref{Corollary.03}]
    Fix $K \in \Z$.
    For $m \geq 1$, we define $V_m := (-m,m]_{\R}^d \times [K+1,K+m]_{\Z}$.
    Let $U_1 := V_1$ and $U_m := V_m \setminus V_{m-1}$ for $m \geq 2$.
    Then, $\{U_m\}_{m=1}^\infty \subset \CU$ forms a collection of mutually disjoint subsets of $\Rd \times \Z$, and thus we have
    \begin{equation*}
        \sum_{m=1}^\infty \Phi_n(U_m) = \Phi_n(\Rd \times [K+1,\infty)_{\Z}).
    \end{equation*}
    Furthermore, Theorem~\ref{Theorem.03} yields that for any $M \geq 1$,
    \begin{equation*}
        \sum_{m=1}^M \Phi_n(U_m) = \Phi_n(V_M) \xrightarrow{d} \mathrm{Po}(\Lambda(V_M)).
    \end{equation*}
    Hence, by the lemma in Appendix~\ref{Appendix.02}, it follows that
    \begin{equation*}
        \Phi_n(\Rd \times [K+1,\infty)_{\Z}) = \sum_{m=1}^\infty \Phi_n(U_m) \xrightarrow{d} \mathrm{Po}(\Lambda(\Rd \times [K+1,\infty)_{\Z})).
    \end{equation*}
    Therefore, we obtain
    \begin{equation*}
        \begin{aligned}
            \Prob[\Delta_n - k_n \leq K] &= \Prob[\Phi_n(\Rd \times [K+1,\infty)_{\Z}) = 0]\\
            &\xrightarrow{n \to \infty} \Prob[\mathrm{Po}(\Lambda(\Rd \times [K+1,\infty)_{\Z})) = 0]\\
            &= \exp(-(-W_0(-e^{-1-1/b}))^{K+1}),
        \end{aligned}
    \end{equation*}
    which completes the proof.
\end{proof}

\begin{proof}[Proof of Corollary~\ref{Corollary.04}]
    Fix $x \in \R$.
    For $m \geq 1$, we define $V_m := (-m,m]_{\R}^d \times (x,x+m]_{\R}$.
    Let $U_1 := V_1$ and $U_m := V_m \setminus V_{m-1}$ for $m \geq 2$.
    Then, $\{U_m\}_{m=1}^\infty \subset \CU$ forms a collection of mutually disjoint subsets of $\Rd \times \R$, and thus we have
    \begin{equation*}
        \sum_{m=1}^\infty \Phi_n(U_m) = \Phi_n(\Rd \times (x,\infty)_{\R}).
    \end{equation*}
    Furthermore, Theorems~\ref{Theorem.06} and \ref{Theorem.07} yield that for any $M \geq 1$,
    \begin{equation*}
        \sum_{m=1}^M \Phi_n(U_m) = \Phi_n(V_M) \xrightarrow{d} \mathrm{Po}(\Lambda(V_M)).
    \end{equation*}
    Hence, by the lemma in Appendix~\ref{Appendix.02}, it follows that
    \begin{equation*}
        \Phi_n(\Rd \times (x,\infty)_{\R}) = \sum_{m=1}^\infty \Phi_n(U_m) \xrightarrow{d} \mathrm{Po}(\Lambda(\Rd \times (x,\infty)_{\R})).
    \end{equation*}
    Therefore, we obtain
    \begin{equation*}
        \begin{aligned}
            \Prob[(\Delta_n - a_n)/b_n \leq x] &= \Prob[\Phi_n(\Rd \times (x,\infty)_{\R}) = 0]\\
            &\xrightarrow{n \to \infty} \Prob[\mathrm{Po}(\Lambda(\Rd \times (x,\infty)_{\R})) = 0]\\
            &= \exp(-\exp(-x)),
        \end{aligned}
    \end{equation*}
    which completes the proof.
\end{proof}

\appendix

\section{Tail estimates and an expectation formula}\label{Appendix.01}
We summarize several fundamental tail estimates and expectation formulas for functionals defined on point configurations that are frequently employed in the study of random geometric graphs.
For their proofs, we refer the reader to \cite[Section~1]{MR1986198}.

Define the function $H : [0,\infty) \to [0,\infty)$ by
\begin{equation*}
    H(t) := 1-t+t\log t \quad (t \in (0,\infty))
\end{equation*}
and $H(0) = 1$ with the convention $0 \log 0 = 0$.
$H(1) = 0$ is the unique zero of $H$, and it holds that $H(t) \geq t$ for all $t \geq 7$.
\begin{lemma}
    Assume that $n, k \in \N$ and $p \in (0,1)$.
    If $k \geq np$, then
    \begin{equation*}
        \Prob[\mathrm{Bi}(n,p) \geq k] \leq \exp\bigg(-np H\bigg(\frac{k}{np}\bigg)\bigg)
    \end{equation*}
    holds and if $k \leq np$,
    \begin{equation*}
        \Prob[\mathrm{Bi}(n,p) \leq k] \leq \exp\bigg(-np H\bigg(\frac{k}{np}\bigg)\bigg)
    \end{equation*}
    holds.
\end{lemma}
\begin{lemma}
    Assume that $\lambda \in (0,\infty)$ and $k \in \N$.
    Then, if $k \geq \lambda$,
    \begin{equation*}
        \Prob[\mathrm{Po}(\lambda) \geq k] \leq \exp\bigg(-\lambda H\bigg(\frac{k}{\lambda}\bigg)\bigg)
    \end{equation*}
    holds and if $k \leq \lambda$,
    \begin{equation*}
        \Prob[\mathrm{Po}(\lambda) \leq k] \leq \exp\bigg(-\lambda H\bigg(\frac{k}{\lambda}\bigg)\bigg)
    \end{equation*}
    holds.
\end{lemma}
\begin{lemma}
    For all $\lambda \in (0,\infty)$ large enough,
    \begin{equation*}
        \Prob[\mathrm{Po}(\lambda) \geq \lambda + \lambda^{3/4}/2] \leq \exp(-\lambda^{1/2}/9), \quad \Prob[\mathrm{Po}(\lambda) \leq \lambda - \lambda^{3/4}/2] \leq \exp(-\lambda^{1/2}/9)
    \end{equation*}
    holds.
\end{lemma}

For measurable functions on the product measurable space $\BN(\Rd) \times \BN(\Rd)$, we have the following expectation formula, which is a fundamental result in the Palm theory for Poisson processes:
\begin{lemma}
    Assume that $\lambda \in (0,\infty)$ and $j \in \N$.
    Let $h : \BN(\Rd) \times \BN(\Rd) \to [0,\infty]$ be a measurable function such that for $\CY, \CX \in \BN(\Rd)$, $h(\CY,\CX) = 0$ unless $\CY (\Rd) = j$.
    Then,
    \begin{equation*}
        \E\bigg[\sum_{\CY \subset \CP_{\lambda}} h(\CY, \CP_{\lambda})\bigg] = \frac{\lambda^j}{j!} \E[h(\CX_j', \CX_j' + \CP_{\lambda})]
    \end{equation*}
    holds, where $\CX_j'$ is an independent copy of $\CX_j$.
\end{lemma}

\section{Poisson approximation for the full series}\label{Appendix.02}
We state a lemma concerning the Poisson approximation for infinite sums of integer-valued random variables as $n \to \infty$, via truncation.
\begin{lemma}\label{Lemma.07}
    Suppose that $\{\xi_{j,n}\}_{j,n \in \N}$ is a collection of non-negative integer-valued random variables that satisfies
    \begin{enumerate}
        \item[(1)] for each $J \in \N$, $\beta_J := \lim_{n \to \infty} \E[\sum_{j \leq J} \xi_{j,n}]$ exists and $\beta_J \to \lim_{n \to \infty} \E[\sum_{j \in \N} \xi_{j,n}] =: \beta < \infty$ as $J \to \infty$,
        \item[(2)] for each $J \in \N$, it holds that $\sum_{j \leq J} \xi_{j,n} \xrightarrow{d} \mathrm{Po}(\beta_J)$ as $n \to \infty$.
    \end{enumerate}
    Then, it holds that $\sum_{j \in \N} \xi_{j,n} \xrightarrow{d} \mathrm{Po}(\beta)$ as $n \to \infty$.
\end{lemma}

\begin{proof}
    Set $S_n := \sum_{j \in \N} \xi_{j,n}$, $S_n^J := \sum_{j \leq J} \xi_{j,n}$ and let $\varphi_n$ and $\varphi$ be the characteristic functions of $S_n$ and $\mathrm{Po}(\beta)$, respectively.
    Then, for $t \in \R$, it holds that
    \begin{equation*}
        |\varphi_n(t) - \varphi(t)| \leq |\E[\exp(itS_n)] - \E[\exp(itS_n^J)]| + |\E[\exp(itS_n^J)] - \exp(\beta(e^{it} - 1))|.
    \end{equation*}
    The first term on the right-hand side is bounded above by
    \begin{equation*}
        \E[|\exp(itS_n) - \exp(itS_n^J)|] = \E[|\exp(it(S_n-S_n^J)) - 1|].
    \end{equation*}
    Since this integrand vanishes whenever $S_n-S_n^J = 0$, it is bounded above by
    \begin{equation*}
        2\Prob[S_n-S_n^J > 0] \leq 2\E[S_n-S_n^J].
    \end{equation*}
    Thus, from Poisson approximation for $S_n^J$, it follows that
    \begin{equation*}
        \limsup_{n \to \infty} |\varphi_n(t) - \varphi(t)| \leq 2(\beta - \beta_J) + |\exp(\beta_J(e^{it} - 1)) - \exp(\beta(e^{it} - 1))|.
    \end{equation*}
    Taking the limit as $J \to \infty$, the right-hand side vanishes, which establishes the assertion.
\end{proof}

\section*{Acknowledgement}
The author would like to thank Professor Ryoki Fukushima for his continuous guidance and fruitful discussions.
He also carefully read the manuscript and provided insightful comments.

\bibliographystyle{plain}
\bibliography{RGG02}

@article {MR4216888,
    AUTHOR = {Chakraborty, Subrata and Chakravarty, Dhrubajyoti and
              Mazucheli, Josmar and Bertoli, Wesley},
     TITLE = {A discrete analog of {G}umbel distribution: properties,
              parameter estimation and applications},
   JOURNAL = {J. Appl. Stat.},
  FJOURNAL = {Journal of Applied Statistics},
    VOLUME = {48},
      YEAR = {2021},
    NUMBER = {4},
     PAGES = {712--737},
      ISSN = {0266-4763,1360-0532},
   MRCLASS = {99-01},
  MRNUMBER = {4216888},
       DOI = {10.1080/02664763.2020.1744538},
       URL = {https://doi.org/10.1080/02664763.2020.1744538},
}

@book {MR3642325,
    AUTHOR = {Kallenberg, Olav},
     TITLE = {Random measures, theory and applications},
    SERIES = {Probability Theory and Stochastic Modelling},
    VOLUME = {77},
 PUBLISHER = {Springer, Cham},
      YEAR = {2017},
     PAGES = {xiii+694},
      ISBN = {978-3-319-41596-3; 978-3-319-41598-7},
   MRCLASS = {60G55 (60G57 60J68)},
  MRNUMBER = {3642325},
MRREVIEWER = {Frank\ Aurzada},
       DOI = {10.1007/978-3-319-41598-7},
       URL = {https://doi.org/10.1007/978-3-319-41598-7},
}

@article {MR2501248,
    AUTHOR = {M\"uller, Tobias},
     TITLE = {Two-point concentration in random geometric graphs},
   JOURNAL = {Combinatorica},
  FJOURNAL = {Combinatorica. An International Journal on Combinatorics and
              the Theory of Computing},
    VOLUME = {28},
      YEAR = {2008},
    NUMBER = {5},
     PAGES = {529--545},
      ISSN = {0209-9683,1439-6912},
   MRCLASS = {05C80 (05C15 60C05 60D05)},
  MRNUMBER = {2501248},
       DOI = {10.1007/s00493-008-2283-3},
       URL = {https://doi.org/10.1007/s00493-008-2283-3},
}

@article{otsuka2026limit,
  title={Limit distributions of the threshold radius for the maximum degree and the associated point configurations in random geometric graphs},
  author={Otsuka, Junpei},
  journal={arXiv preprint arXiv:2604.24624},
  year={2026}
}

@book {MR1986198,
    AUTHOR = {Penrose, Mathew},
     TITLE = {Random geometric graphs},
    SERIES = {Oxford Studies in Probability},
    VOLUME = {5},
 PUBLISHER = {Oxford University Press, Oxford},
      YEAR = {2003},
     PAGES = {xiv+330},
      ISBN = {0-19-850626-0},
   MRCLASS = {60-02 (05C80 60D05)},
  MRNUMBER = {1986198},
MRREVIEWER = {Ilya\ S.\ Molchanov},
       DOI = {10.1093/acprof:oso/9780198506263.001.0001},
       URL = {https://doi.org/10.1093/acprof:oso/9780198506263.001.0001},
}

@article {MR4993794,
    AUTHOR = {Penrose, Mathew D. and Yang, Xiaochuan},
     TITLE = {Fluctuations of the connectivity threshold and largest
              nearest-neighbour link},
   JOURNAL = {Ann. Appl. Probab.},
  FJOURNAL = {The Annals of Applied Probability},
    VOLUME = {35},
      YEAR = {2025},
    NUMBER = {6},
     PAGES = {3906--3941},
      ISSN = {1050-5164,2168-8737},
   MRCLASS = {60D05 (05C62 60F05 60G70)},
  MRNUMBER = {4993794},
MRREVIEWER = {D.\ Yogeshwaran},
       DOI = {10.1214/25-AAP2210},
       URL = {https://doi-org.tsukuba.idm.oclc.org/10.1214/25-AAP2210},
}
\end{document}